\documentclass[a4paper,english]{article}%[15pt,psamsfonts]
\usepackage{amssymb,amsmath}
\makeindex

\input xy  \xyoption{all}

\newtheorem{deff}{Definition}[section]
\newtheorem{lemma}[deff]{Lemma}
\newtheorem{theorem}[deff]{Theorem}
\newtheorem{corollary}[deff]{Corollary}
\newtheorem{proposition}[deff]{Proposition}

\newtheorem{facts}[deff]{Facts}
\newtheorem{em-example}[deff]{Example}
\newtheorem{em-def}[deff]{Definition}        % definition(auxiliary)
\newtheorem{em-remark}[deff]{Remark}         % remark(auxiliary)
\newtheorem{em-question}[deff]{Question}

\newtheorem{claim}[deff]{Claim}

\newenvironment{example}{\begin{em-example} \em }{ \end{em-example}}
\newenvironment{definition}{\begin{em-def} \em  }{ \end{em-def}}
\newenvironment{remark}{\begin{em-remark} \em }{\end{em-remark}}
\newenvironment{question}{\begin{em-question}\em }{\end{em-question}}
\newenvironment{proof}{\noindent {\it Proof}.}{\QED \smallskip}
\newcommand{\ol}[1]{\overline{#1}}
\newcommand{\eps}{\varepsilon}
\newcommand{\wt}[1]{\widetilde{#1}}
\newcommand{\dis}{\displaystyle}

\newcommand{\fP}{\mathfrak {P}}
\newcommand{\fZ}{\mathfrak {Z}}

\newcommand\QED{\hfill QED \medskip}

\newcommand{\bcal}{{\cal B}}

\newcommand{\ncal}{{\cal {N}}}

\newcommand{\tcal}{\cal {T}}

\def\hull#1{\langle#1\rangle}
\def\eps{\varepsilon}

\def\ker{\mathop{\rm ker}}

\def\:{\nobreak \hskip .1111em\mathpunct {}\nonscript \mkern
-\thinmuskip {:}\hskip .3333emplus.0555em\relax}

\def\nbd{neighborhood}
\catcode`\@=12
\def\T{{\mathbb T}}

\def\Z{{\mathbb Z}}
\def\N{{\mathbb N}}
\def\R{{\mathbb R}}
\def\Q{{\mathbb Q}}
\def\P{{\mathbb P}}

\def\Prm{\P}

\def\cont{\mathfrak c}

\begin{document}
\title{{ Locally minimal topological groups} \thanks{The first named author was partially supported by MTM 2008-04599. The third author   was partially supported by SRA, grants P1-0292-0101 and J1-9643-0101. The other authors were partially supported by MTM 2006-03036 and FEDER funds.}
}
\author{
Lydia Au\ss enhofer \\{\small\em Institut f\"{u}r Algebraische Geometrie,} {\small\em Welfengarten 1
D-30167 Hannover,} {\small\em e-mail: aussenhofer@maphy.uni-hannover.de}\\
 M. J. Chasco \\ {\small\em Dept. de F\'{\i}sica y Matem\'{a}tica Aplicada,} 
{\small\em Universidad  de Navarra,} {\small\em e-mail: mjchasco@unav.es}\\
 Dikran Dikranjan \\{\small\em Dipartimento di Matematica e Informatica,} {\small\em Universit\`{a} di Udine,} {\small\em e-mail: dikranja@dimi.uniud.it}\\
 Xabier Dom\'{\i}nguez\\ {\small\em Departamento de M\'{e}todos Matem\'{a}ticos y de
Representaci\'{o}n,} {\small\em Universidad  de A Coru\~{n}a,} {\small\em e-mail: xdominguez@udc.es}
}
\date{October 15, 2009}
\maketitle
\begin{abstract}
%%%%%%%%%%%%%%%%%%%%%%%
%footnote{ Algunas cosas a tener en cuenta antes de mandar el art\'{\i}culo a publicar:
%begin{enumerate}
%item Cuando tengamos el macro de la revista, poner las direcciones. Recordar que la de Lydia es: Instituto de
%atem\'aticas - Universidad Jaume I - E-12071 Castell\'on.
%item Tambi\'en cuando sepamos la revista: Unificar el estilo de las referencias.
%item A riesgo de ser pesado, las referencias al libro de Dikran-Prodanov-Stoyanov
%eber\'{\i}an referir al resultado concreto o al menos a una p\'agina. (Algunos he encontrado, otros
%o; entonces he marcado el sitio con $\spadesuit $ and a corresponging footnote)
%item Comprobar las referencias cruzadas, porque algunos resultados han ``ascendido" de categor\'{\i}a.
%item No olvidarse de incluir a Elena en los agradecimientos.
%\item Decidirse entre ``neighborhood" y ``neighbourhood". Hecho: neighborhood (hice 9 cambios).
%end{enumerate}}

%%%%%%%%%%%%%%%%%%%%%%%
The aim of this paper is to go deeper into the study of local minimality and its connection
to some naturally related properties.
A Hausdorff topological group $(G,\tau)$ is called locally minimal if there exists a neighborhood $U$ of $0$ in $\tau$ such that $U$ fails to
 be a neighborhood of zero in any Hausdorff group topology on $G$ which is
 strictly coarser than $\tau.$ Examples of locally minimal groups are all
subgroups of Banach-Lie groups, all locally compact groups and all minimal groups.

%We extend the minimality criterion for dense subgroups of compact
%groups to local minimality.  Using this criterion
% we characterize the compact
%abelian groups containing dense countable locally minimal subgroups,
%as well as those
%  containing dense locally minimal subgroups of countable free-rank.
%  We  also characterize the compact abelian groups whose torsion part is dense
%  and locally minimal.

Motivated by the fact that locally compact NSS groups are Lie groups, we study the connection between local minimality
and the NSS property, establishing that under certain conditions, locally minimal  NSS groups are metrizable.

A symmetric subset of an abelian group  containing zero is said to be a GTG
set if it generates a group topology in an analogous way as convex and symmetric
subsets are unit balls for pseudonorms on a vector space. We consider topological
groups which have a neighborhood basis at zero consisting of GTG sets. Examples of
these locally GTG groups are: locally pseudo--convex spaces, groups uniformly free from small subgroups (UFSS groups) and locally compact abelian groups.
%$\spadesuit $\footnote{New; infinite products of locally GTG groups need not be loc. GTG}.
The precise relation between these classes of groups is obtained: a topological abelian group is UFSS if and only if it is
locally minimal, locally GTG and NSS.
We develop a universal construction of GTG sets  in arbitrary non-discrete metric abelian groups, that
generates a strictly finer non-discrete UFSS topology and we characterize the metrizable abelian groups admitting a strictly finer non-discrete
UFSS group topology.

Unlike the minimal topologies, the locally minimal ones are always available on ``large" groups. To support this line,
we prove  that a bounded abelian group $G$ admits
a non-discrete locally minimal and locally GTG group topology iff $|G|\geq \cont$.

%%%%%%%%%%%%%%%%%%%%%%%%%%%%


MSC 22A05, 22B05, 54H11, 52A30
\end{abstract}

%\tableofcontents

%%%%%%%%%%%%%%%%%%%%%%%%%%%%
\section{Introduction}
%\subsection{Main results}
%%%%%%%%%%%%%%%%%%%%%%%%%%%%%%%

Minimal topological spaces have been largely studied in the literature (\cite{BPS}). Minimal topological groups were introduced  independently by Choquet, Do\" \i tchinov \cite{Do1} and Stephenson \cite{St}: a Hausdorff topological group $(G,\tau)$ is called minimal
 if there exists no Hausdorff group topology on $G$ which is strictly coarser than $\tau$.
The major problem that determined the theory of minimal abelian groups
was establishing {\em precompactness} of the abelian minimal groups (Prodanov-Stoyanov's theorem
\cite[Theorem 2.7.7]{DPS}; for recent advances in this field see
\cite{D,D1,DPS}).

Generalizations of minimality were recently proposed by various authors. Relative minimality and co-mini\-ma\-lity were introduced by Megrelishvili in \cite{meg}
(see also \cite{DMeg, Menny}). The  notion of local minimality (see Definition \ref{def:loc:min}) was introduced by Morris and  Pestov in \cite{MP} (see also  Banakh \cite{TaB}). A stronger version of this notion was used in \cite{DM} to  characterize the locally compact subgroups of infinite products of locally compact groups.

   We start Section \ref{local_minimality} with some permanence properties of local minimality (with respect to taking closed or open subgroups).
We prove in Theorem \ref{nw} that $nw(G)=w(G)$ for every locally minimal group (in particular, all countable locally minimal groups are metrizable). Subsection \S \ref{subsection_NSS_groups} is dedicated to the NSS groups. Let us recall, that a  topological group $(G,\tau)$ is called {\em NSS group} (resp.,  {\em NSnS group}) if   a suitable zero neighborhood contains only the trivial (resp., normal) subgroup. The relevance of the NSS property comes from the fact that it characterizes the Lie groups within the class of locally compact groups.
Since local minimality generalizes local compactness, it is quite natural to investigate local minimality combined with the NSS property.
 It turns out that  locally minimal abelian NSS groups are metrizable (Proposition \ref{loc_min+NSS>metr}), which should be compared with the classical fact that locally compact NSS groups are Lie groups (hence, metrizable).
 We do not know whether ``abelian" can be removed here (cf. Question \ref{ques:NSS}).
 %A crucial role in the theory of minimal groups was played by the so called "minimality criterion", due to Banaschewski  \cite{B}, Stephenson \cite{St} and Prodanov \cite{P1}; namely,  the characterization of those dense subgroups of a minimal group that are minimal as topological
%subgroups (we recall it in Theorem \ref{Min_Crit}). In \S \ref{loc_ess_loc_min} we obtain a counterpart of this criterion (Theorem \ref{crit}) for local minimality based on the new notion of a
%{\em locally essential subgroup} (see Definition \ref{LocEsse}).  Using this local minimality criterion we characterize the compact abelian groups containing dense countable locally minimal subgroups, as well as those  containing dense locally minimal subgroups of countable
%free-rank. We also characterize the compact abelian groups whose torsion part is dense and locally minimal. All these results are based on  known counterparts in the case of minimal groups.
 %%%%%%%%%%%%%%%%%%%%%%%%%%%%

%%%%%%%%%%%%%%%%%%%%%%%%%%%%
Section \ref{section_ufss} is dedicated to a property, introduced by Enflo \cite{E} that simultaneously strengthens local minimality and the NSS property.
 A Hausdorff topological group is UFSS (Uniformly Free from Small Subgroups) if its topology is generated by a single neighborhood of zero in a natural analogous way as the unit ball of a normed space determines its topology (a precise definition is given in \ref{defenflo} below).
In Proposition \ref{Min+NSS} we show that locally minimal NSnS precompact groups are UFSS (hence minimal NSS abelian groups are UFSS).
Local minimality presents a common generalization of local compactness, minimality and UFSS. Since the latter property is not sufficiently
studied, in contrast with the former two, we dedicate \S \ref{perm_ufss} to a detailed study of the permanence properties of this remarkable class. We show in Proposition \ref{perm_prop_3}
that UFSS is stable under taking subgroups, extensions (in particular, finite products), completions and local isomorphisms.
%%%%%%%%%%%%%%%%%%%%%%%%%%%%

%%%%%%%%%%%%%%%%%%%%%%%%%%%%
In \S \ref{Lydia_Chapter} we introduce the concept of a GTG set that, roughly speaking,
is a symmetric subset $U$ of a group $G$ containing 0, with an appropriate
convexity-like property (i.~e., these sets are generalizations of the symmetric convex sets in real vector spaces,
see Definition \ref{definicion_gtg}).  A topological group is called locally GTG, if it has a base of \nbd s of 0 that are GTG sets.
Since locally precompact abelian groups, as well as UFSS groups, are locally GTG, this explains the importance of this new
class. On the other hand, minimal abelian groups are precompact,  so minimal abelian groups are both locally minimal and locally GTG.
%In other words, the class of locally minimal and locally GTG groups covers, under the same umbrella, all three classes of topological abelian groups: locally
%compact groups, minimal groups and UFSS groups (so in particular, all normed spaces). The precise relation is explained by Theorem \ref{MJX}:
We prove in Theorem \ref{MJX} that a Hausdorff abelian topological group is UFSS iff it is locally minimal, NSS and locally GTG.
According to a theorem of  Hewitt \cite{H}, the usual topologies on the
%%%%%%%%%%%%%%%%%%%%%%
%$\clubsuit$\footnote{\sl I propose to insert here: ``usual topologies on the"; done; strictly stronger included in
%the second part or the sentence.}
%%%%%%%%%%%%%%%%%%%%
group $\T$ and the group $\R$ have the property that the only strictly finer locally compact group
topologies are the discrete topologies. Since locally minimal locally GTG topologies generalize the locally compact group
topologies, it would be natural to ask whether the groups $\T$ and $\R$ admit stronger non-discrete locally minimal locally GTG topologies.
In Corollary \ref{Hewitt} we give a strongly positive answer to this question for the large class of all non-totally
disconnected locally compact metrizable abelian groups and for the stronger property of UFSS topologies.
To this end we develop, in Theorem \ref{Lydia_Theorem}, a universal construction of GTG sets  in arbitrary non-discrete metric abelian groups, that
generates a strictly finer non-discrete UFSS topology
%compact totally disconnected GTG sets in arbitrary complete  non-discrete metric abelian groups
 .

The description of the {\em algebraic structure} of locally minimal abelian groups seems to be an important problem. Its solution for the class of compact groups by the end of the fifties of the last century brought a significant development of the theory of infinite abelian groups.
This line was followed later also in the theory of minimal groups, but here the problem is still open even if
solutions in the case of  many smaller classes of abelian groups are available
(\cite[\S\S 4.3, 7.5]{D}, \cite[chapter 5]{DPS}).
Unlike the minimal topologies, the locally minimal ones are always available on ``large" groups. To support this line,
we prove in Theorem \ref{NewThm1} that a bounded abelian group $G$ admits
a non-discrete locally minimal and locally GTG group topology iff $|G|\geq \cont$
 (and this occurs precisely when $G$ admits a non-discrete locally compact
 group topology).
Analogously,  in another small group (namely, $\Z$), the non-discrete
locally minimal and locally GTG group topologies are not much more than the minimal ones (i.~e., they are either UFSS
 or have an open minimal subgroup, see Example \ref{Ex_LA1}). This line will be pursued further and in more detail in the forthcoming paper
\cite{AlmMin} where we study also the locally minimal groups that can be obtained as extensions of a minimal group via a UFSS
quotient group.

In the next diagram we collect all implications between all properties introduced so far:

%\vskip0.15in

$$%\xymatrix
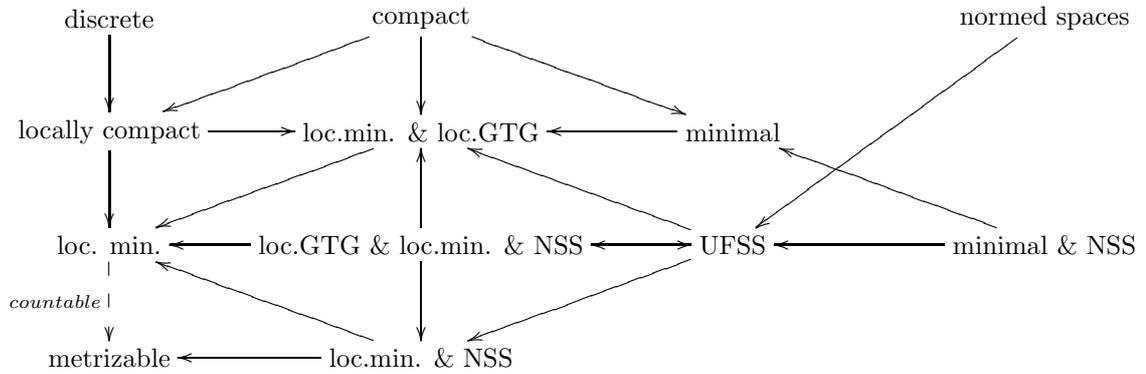
{\xymatrix@!0@C4.1cm@R=1.5cm{
   \mbox{discrete} \ar@{->}[d]
 % \ar@{->}[ddrr]
   &  \mbox{compact}\ar@{->}[d]\ar@{->}[dl] \ar@{->}[dr]& &\mbox{normed spaces}
   \ar@{->}[ddl]\\
%. & .& .& .\\
  \mbox{locally compact} \ar@{->}[d] \ar@{->}[r]& \mbox{loc.min. \& loc.GTG}
  \ar@{->}[dl] & \mbox{minimal} \ar@{->}[l]& \\
   \mbox{loc. min.} \ar@{-->}[d]_{countable} %\ar@{->}[r]% \ar@{->}[dr]
  &\mbox{loc.GTG \& loc.min. \& NSS} \ar@{->}[u]\ar@{->}[l]\ar@{->}[d]\ar@{->}[r]
  %\ar@{-->}[d]_{?}
 % \ar@{->}[ddl]
 &  \mbox{UFSS} \ar@{->}[ul]\ar@{->}[l] \ar@{->}[ld]  &
     \mbox{minimal \& NSS}\ar@{->}[l]
     %_{abelian}
     \ar@{->}[ul]
    %\ar@{->}[d]
      \\%
    \mbox{metrizable} &\mbox{loc.min. \& NSS}\ar@{->}[ul]\ar@{->}[l]%_{abelian}
    &
    %\mbox{loc.min.\& NSS}  \ar@{->}[r]\ar@{->}[d]\ar@{-->}[l]_{\phantom{M}?} & \mbox{ NSS }
    \\
 %   \mbox{loc.minimal} \ar@{-->}[rr]^{\textrm{countable}} & & \mbox{metrizable}  & \\ % \ar@{->}[lu] %\ar@{->}[l]  \\
    } }$$

\vskip0.2in

All the implications denoted by a solid arrow are true for arbitrary abelian groups, those
that require some additional condition on the group are
given by dotted arrows accompanied by the additional condition  in question.
%Obviously, since (locally) GTG was defined only for abelian groups, the implications
%starting from or ending up in the central two boxes ``loc.GTG \& loc.min. \& NSS"  and
%``loc.GTG \& loc.min." should have been given as dotted and accompanied by ``abelian". This was not
%done in order to avoid useless complication of the diagram. Anyway, for the benefit of
%the reader we
We give separately in the next diagram only those arrows that are valid for all, not necessarily abelian, topological groups.
%\NB \footnote{I hope these diagram leave now no more question about what is valid in the abelian case, etc.
%Anyway, please, feel free to ask further questions. \par

%I have scaled these two diagrams in a different way to show how the difference
%may look like. We can make the second a bit smaller to avoid big divergence.
%{\sl OK but I would prefer simply one diagram for the abelian case and another one for the general case. Besides that, the arrow from ``locally compact" to ``metrizable" on the second diagram should dissapear.} }

%some of them also for non abelian groups.

\vskip0.2in
%\begin{center}
%\hbox{
$\phantom{MMMMM}%\xymatrix
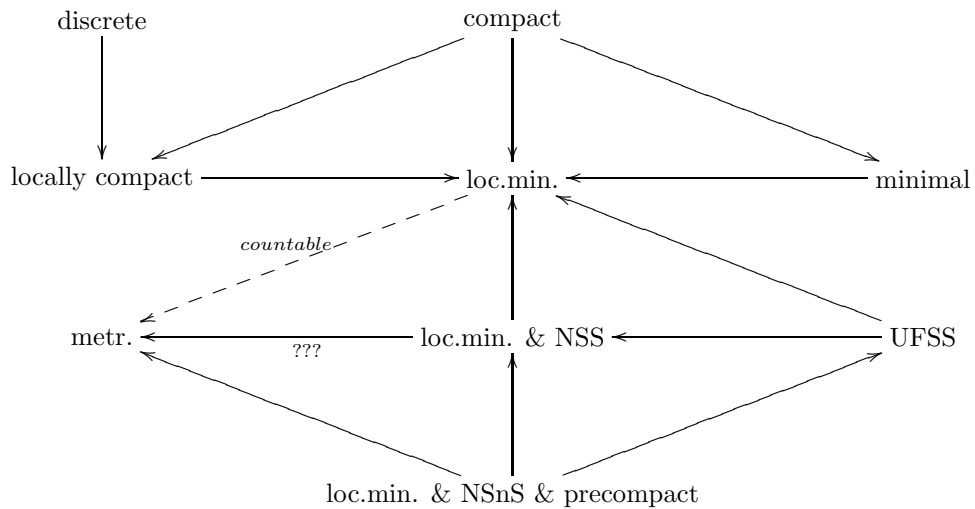
{\xymatrix@!0@C5.4cm@R=2.1cm{
   \mbox{discrete} \ar@{->}[d]
 % \ar@{->}[ddrr]
   &  \mbox{compact}\ar@{->}[d]\ar@{->}[dl] \ar@{->}[dr]& & & \\
   %\mbox{normed spaces}\ar@{->}[ddl]\\
%. & .& .& .\\
  \mbox{locally compact} %\ar@{->}[d]
  \ar@{->}[r]& \mbox{loc.min.}
\ar@{-->}[dl]_{countable}
  & \mbox{minimal} \ar@{->}[l]& \\
   \mbox{metr.}
   % \ar@{-->}[d]_{countable} %\ar@{->}[r]% \ar@{->}[dr]
  &\mbox{loc.min. \& NSS} \ar@{->}[u] \ar@{->}[l]^{???}
  %\ar@{->}[l] \ar@{->}[d]\ar@{->}[r]
  %\ar@{-->}[d]_{?}
 % \ar@{->}[ddl]
 &  \mbox{UFSS} \ar@{->}[ul]\ar@{->}[l] %\ar@{->}[ld]
  & \\
 %    \mbox{minimal \& NSS}\ar@{->}[l]\ar@{->}[ul]
    %\ar@{->}[d]
    %\mbox{metrizable}
    &\mbox{loc.min. \& NSnS \& precompact}\ar@{->}[ur]\ar@{->}[u]\ar@{->}[ul]&
    %\mbox{loc.min.\& NSS}  \ar@{->}[r]\ar@{->}[d]\ar@{-->}[l]_{\phantom{M}?} & \mbox{ NSS }
    \\
 %   \mbox{loc.minimal} \ar@{-->}[rr]^{\textrm{countable}} & & \mbox{metrizable}  & \\ % \ar@{->}[lu] %\ar@{->}[l]  \\
    } }$
%}

%\end{center}

\vskip0.2in

%
%$$%\xymatrix
%{\xymatrix@!0@C4.2cm@R=1.6cm{
%   \mbox{discrete}  \ar@{->}[d]  \ar@{->}[ddrr]  &  \mbox{compact}\ar@{->}[dl] \ar@{->}[dr]& &\\
%%. & .& .& .\\
%  \mbox{locally compact} \ar@{->}[dr]  \ar@{->}[d]& & \mbox{minimal}\ar@{->}[dl] & \mbox{minimal \& NSS}\ar@{->}[d]\ar@{->}[dl]\ar@{->}[l] \\
%   \mbox{complete loc.q.c. loc.min.} \ar@{->}[r]% \ar@{->}[dr]
%    \ar@{->}[d]&\mbox{almost minimal} \ar@{-->}[d]_{?}\ar@{->}[ddl] &  \mbox{UFSS}\ar@{->}[r]\ar@{->}[l]\ar@{->}[dr] \ar@{->}[d]\ar@{->}[ld]  &
%    \mbox{minimal \& NSS}\ar@{->}[d]\ar@{->}[ul]\ar@{->}[l]
%    \mbox{loc.GTG \& loc.min. \& NSS}\ar@{->}[l] \ar@{->}[d]   \\%
%    \mbox{loc.q.c. \& loc.min.}\ar@{->}[r]\ar@{->}[d]  &\mbox{loc.min. \& loc.GTG}\ar@{->}[dl]&\mbox{loc.min.\& NSS}  \ar@{->}[r]\ar@{->}[d]\ar@{-->}[l]_{\phantom{M}?} & \mbox{ NSS }\\
%    \mbox{loc.minimal} \ar@{-->}[rr]^{\textrm{countable}} & & \mbox{metrizable}  & \\ % \ar@{->}[lu] %\ar@{->}[l]  \\
%    } }$$

\paragraph{Notation and terminology}

The subgroup generated by a subset $X$ of a group $G$ is denoted by
$\hull{X}$, and $\hull{x}$ is the cyclic subgroup of $G$ generated
by an element $x\in G$. The abbreviation $K\leq G$ is used to denote
a subgroup  $K$ of $G$.

We use additive notation for a not
necessarily abelian group, and denote by $0$ its neutral element.
We denote by $\N$, $\N_0$ and $\Prm$ the sets of positive natural numbers, non-negative
integers
and primes, respectively; by $\Z$ the integers, by $\Q$ the
rationals, by $\R$ the reals, and by $\T$ the unit circle group
which is identified with $\R/\Z$. The cyclic group of order $n>1$ is
denoted by $\Z(n)$. For a  prime $p$ the symbol $\Z(p^\infty)$
stands for the  quasicyclic $p$-group and $\Z_p$ stands for the
$p$-adic integers.

The {\it torsion part\/} $t(G)$ of an abelian group $G$ is the set
$\{x\in G: nx=0\ {\rm for\ some}\ n\in\N\}$. Clearly, $t(G)$ is a
subgroup of $G$.  For any $p\in {\mathbb P}$, the {\em $p$-primary component}
$G_p$ of $G$ is the subgroup of $G$ that consists of all $x\in G$
satisfying $p^nx=0$ for some positive integer $n$. For every
$n\in\N$, we put $G[n]=\{x\in G: nx=0\}$. We say that $G$ is {\it bounded}
if $G[n]=\{0\}$ for
some $n\in \N$. If $p\in {\mathbb P},$ the
$p$-{\it rank\/} of $G$, $r_p(G)$, is defined as the cardinality of
a maximal independent subset of $G[p]$ (see \cite[Section
4.2]{Rob}). The group $G$ is {\it divisible} if $nG=G$ for every
$n\in \N$, and {\it reduced}, if it has no divisible subgroups
beyond $\{0\}$.  The {\em free rank\/} $r(G)$ of the group $G$ is the
cardinality of a maximal independent subset of $G$. The {\em socle} of
$G,$ $Soc(G),$ is the subgroup of $G$ generated by all elements of prime
order, i. e. $Soc(G)=\bigoplus_{p\in {\mathbb P}}\,G[p].$

%Throughout this paper all topological groups are assumed to be
%Hausdorff, unless otherwise explicitly stated.
We denote by ${\cal
V}_\tau(0)$ (or simply by ${\cal V}(0)$) the filter of neighborhoods
of the neutral element $0$ in a topological group $(G, \tau)$. Neighborhoods
are not necessarily open.

For a topological group $G$ we denote by $\widetilde{G}$ the Ra\u \i
kov completion of $G$. We recall here that a group $G$ is {\it
precompact\/} if $\widetilde{G}$ is compact (some authors prefer the
term ``totally bounded").

We say a topological group $G$ {\em is linear} or {\em is linearly
topologized} if it has a neighborhood basis at $0$ formed by open subgroups.

The cardinality of the continuum $2^{\omega}$ will be also denoted by
$\cont$. The {\em weight} of a topological space $X$ is the minimal
cardinality of a basis for its topology; it will be denoted by
$w(X).$ The {\em netweight} of $X$ is the minimal cardinality of a
network in $X$ (that is, a family ${\mathcal N}$ of subsets of $X$
such that for any $x\in X$ and any open set $U$ containing $x$ there
exists $N\in {\mathcal N}$ with $x\in N\subseteq U$). The netweight
of a space  $X$ will be denoted by $nw(X)$.
The {\em
pseudocharacter} $\psi(X,x)$ of a space $X$ at a point $x$ is the
minimal cardinality of a family of open neighborhoods of $x$ whose
intersection is $\{x\};$ if $X$ is a homogeneous space, its
pseudocharacter is the same at every point and we denote it by
$\psi(X).$
The {\em Lindel\"of number} $l(X)$ of a space $X$ is the
minimal cardinal $\kappa$ such that any open cover of $X$ admits a
subcover of cardinality not greater than $\kappa.$

By a  {\em character} on an abelian topological group $G$ it is
commonly understood a continuous homomorphism from $G$ into the unit
circle group  $\mathbb{T}$. %%%%%%%%%%%%%%%%%%%%%%%%%%%%

 Let $U$ be a symmetric subset of a group $(G,+)$ such that $0\in U,$ and $n\in \N$. We define
 $(1/n)U: = \{x\in G\,:\, kx\in U \;\forall k\in\{1,2,\cdots, n\}\}$ and
 $U_\infty:=\{x \in G\,:\, nx \in U \;\forall n\in \N\}.$

%%%%%%%%%%%%%%%%%%%%%%%%%%%%
Recall that a nonempty subset $U$ of a real vector space is {\em starlike} whenever $[0,1] U \subseteq U.$ Note that if $U$ is starlike and symmetric then $(1/n)U=\displaystyle \frac{1}{n}U$; in general, for symmetric $U$: $\displaystyle(1/n)U=
 \bigcap_{k=1}^n \frac{1}{k}U.$

All unexplained topological terms can be found in \cite{Eng}. For
background on abelian groups, see \cite{Fuc} and \cite{Rob}.

\section{Local minimality }\label{local_minimality}

\subsection{The notion of a locally minimal topological group}

In this section we recall the definition and basic examples of locally
minimal groups, and prove that for locally minimal groups the weight and the netweight coincide.

\begin{definition}\label{def:loc:min} A Hausdorff
topological group $(G, \tau)$ is {\em locally minimal} if there exists a
neighborhood $V$ of $0$ such that whenever $\sigma\leq \tau$  is a Hausdorff
group topology on $G$ such that $V$ is a $\sigma$-neighborhood of $0$, then
$\sigma=\tau$. If we want to point out that the neighborhood $V$ witnesses local
minimality for $(G,\tau)$ in this sense, we say that $(G,\tau)$ is
{\em $V$--locally minimal.}
\end{definition}

\begin{remark}  As mentioned in \cite{DM}, one obtains an equivalent definition replacing ``$V$ is a $\sigma$-neighborhood of $0$" with ``$V$ has a
non-empty $\sigma$-interior" above.

It is easy to see that if local minimality of a group $G$ is witnessed by some $V\in {\cal V}_\tau(0)$, then every smaller $U\in
{\cal V}_\tau(0)$ witnesses local minimality of $G$ as well.
\end{remark}

\begin{example}\label{Exa_LocMin}
Examples for locally minimal groups:
\begin{itemize}
\item[(a)] If $G$ is a minimal topological group, $G$ is locally
minimal [$G$ witnesses local minimality of $G$].
\item[(b)] If $G$ is a locally compact group, $G$ is locally minimal
[every compact neighborhood of zero witnesses local minimality of $G$,
{\rm \cite{DM}}].
   \item[(c)] It is easy to check that a normed space  $(E,\tau)$ with unit
ball $B$ is $B$-locally minimal.
\end{itemize}
\end{example}

We start with some permanence properties of locally minimal groups.

\begin{proposition} \label{locally minimal open subgroup}A group having an open locally minimal
subgroup is locally minimal.
\end{proposition}

\begin{proof} Let $H$ be a locally minimal group witnessed by $U\in{\cal V}_H(0)$ and suppose
that $H$ is an open subgroup of the Hausdorff group $(G,\tau)$.
Then $U$ is a neighborhood of $0$ in $G$. Assume that $\sigma$ is
a Hausdorff group topology on $G$ coarser than $\tau$ such that
$U$ is a neighborhood of $0$ in $(G,\sigma)$. Then $\tau|_H\ge
\sigma|_H$ and since $U$ is a neighborhood of $0$ in
$(H,\sigma|_H)$, we obtain $ \tau|_H=  \sigma|_H$. Since $U$ is a
neighborhood of $0$ in $(G,\sigma)$, the subgroup $H$ is open in
$\sigma$ and hence $\sigma=\tau$.
\end{proof}

In the other direction we can weaken the hypothesis ``open subgroup" to the much weaker ``closed subgroup",
but we need to further impose the restraint on $H$ to be central.

\begin{proposition}\label{subgr}
Let $G$ be a locally minimal group and let $H$ be a closed central subgroup of $G$.
Then $H$ is locally minimal.
\end{proposition}

\begin{proof}
Let $\tau$ denote the topology of $G$ and let $V_0\in {\cal
V}_{(G,\tau)}(0)$ witness local minimality of $(G,\tau)$. Choose
$V_1\in {\cal V}_{(G,\tau)}(0)$ such that $V_1+ V_1\subseteq V_0$.
We show that $V_1\cap H$ witnesses  local minimality of $H$.
Suppose $\sigma$ is a Hausdorff group topology on $H$ coarser than
$\tau|_H$ such that $V_1\cap H$ is $\sigma$-neighborhood of 0. It is easy to verify that the family of sets $(U+V)$
where $U$ is a $\sigma$-neighborhood
 of 0 in $H$ and $V$ is a $\tau$-neighborhood
of 0,  form a neighborhood basis of a group topology $\tau'$ on $G$
which is coarser than $\tau$. Let us prove that $\tau'$ is Hausdorff:
Therefore, observe that for a subset $A\subseteq H$ we have
$\overline{A}^\tau\subseteq \overline{A}^\sigma$, since $H$ is
closed in $\tau$. Hence we obtain $\overline{\{0\}}^{\tau'}=
\bigcap\{U+ V: U\in {\cal  V}_{(H,\sigma)}(0), V\in {\cal
V}_{(G,\tau)}\}=\bigcap_{U\in {\cal V}_{(H,\sigma)}(0)}
\bigcap_{V\in {\cal V}_{(G,\tau)}(0)}U+V= \bigcap_{U\in {\cal
 V}_{(H,\sigma)}(0)}\overline{U}^{\tau}\subseteq
\bigcap_{U\in{\cal V}_{(H,\sigma)}(0)}\overline{U}^\sigma=\{0\}$
since $\sigma$ was assumed to be Hausdorff.

Moreover, if $W\in {\cal V}_\sigma(0)$ such that $W\subseteq V_1\cap H$, then $W+ V_1\subseteq V_0$ implies that $V_0\in {\cal V}_{(G,\tau')}(0)$.
 By the choice of $V_0$ this yields
$\tau'=\tau$. Hence $\sigma=\tau$.
\end{proof}

\begin{corollary}\label{Coro_}
An open central subgroup $U$ of a topological
group $G$ is locally minimal iff $G$ itself is
locally minimal.
\end{corollary}

These results leave open the question on whether ``central" can be omitted in the above
corollary and Proposition \ref{subgr} (see Question \ref{LAST_QUES}).

The question whether the product of two minimal (abelian) groups is again minimal was
answered  negatively by Do\"\i tchinov in \cite{Do1}  where he
proved that $(\Z,\tau_2)\times (\Z,\tau_2)$ is not minimal although the $2$--adic topology $\tau_2$ on the integers is
minimal. We will show in Prop. \ref{Do1loc}  that $(\Z,\tau_2)\times (\Z,\tau_2)$ is not even locally minimal.

  Next we are going to see some cases where metrizability can be deduced
  from local minimality. We start with a generalization to locally minimal
  groups of the following theorem of Arhangel$'$skij: $w(G)=nw(G)$ for
  every minimal group; in particular, every
 minimal group with countable netweight is metrizable. For that we need
 the following result from
 \cite{Arh}:

\begin{lemma}\label{lemarh}
Let $\kappa$ be an infinite cardinal and let $G$ be a topological group with
\begin{itemize}
  \item[(a)] $\psi(G)\leq \kappa$;
  \item[(b)] $G$ has a subset $X$ with $\langle X \rangle =G$ and $l(X)\leq
  \kappa$.
\end{itemize}
Then for every family of neighborhoods $\bcal$ of the neutral element $0$
of $G$ with $|{\bcal}|\leq \kappa$ there exists a coarser group
topology $\tau'$ on $G$ such that $w(G,\tau')\leq \kappa$ and every $U\in
\bcal$ is a $\tau'$-neighborhood \ of $0$.
\end{lemma}

\begin{theorem}\label{nw} For a locally minimal group  $(G,\tau)$ one has
$w(G)=nw(G)$. In particular, every countable locally minimal group is metrizable.
\end{theorem}

\begin{proof} Let $\kappa=nw(G)$ and let ${\mathcal N}$ be a network of $G$
 of size $\kappa$. Then also $\psi(G)\leq \kappa$ as
$$
  \bigcap\{G\setminus \ol{B}:0\not \in \ol{B},\ B\in \ncal\} =\{0\}.
$$
Moreover, the Lindel\" of number $l(G)$ of $G$ is $\leq \kappa$. Indeed, if
 $G=\bigcup_{i\in I}U_i$ and each $U_i$ is a non-empty
open set, then by the definition of a network for every $x\in G$ there exists
$i_x\in I$ and $B_x\in \ncal$ such that $x\in B_x\subseteq
U_{i_x}$. [For $z\in G$ we choose $y\in Y$
such that $B_z=B_y$ and obtain $z\in B_z=B_y\subseteq U_{i_y}$.]  Let
 $\ncal_1=\{B_x: x\in G\}$ and $Y\subseteq G$ such that the assignment
 $Y\to \ncal_1$, defined by  $Y\ni x\mapsto B_x$,
is  bijective. Then $|Y|\leq \kappa$ and $G =\bigcup_{y\in Y}U_{i_y}$.
 This proves $l(G)\leq \kappa$.

To end the proof of the theorem  apply Lemma \ref{lemarh} taking $X=G$
and any family $\mathcal B$  of size $\kappa$ of
$\tau$-neighborhoods of $0$ containing $U$ as a member and witnessing
$\psi(G)\leq \kappa$ (i.e., $\bigcap {\mathcal
B}=\{0\}$). This gives a Hausdorff topology $\tau'\leq \tau$ on $G$
satisfying the conclusion of the lemma.  By the local minimality of
$(G,\tau)$ we conclude $\tau'=\tau$. In particular, $w(G,\tau)\leq \kappa$.
Since always $nw(G)\leq w(G)$, this proves the required
equality $w(G)=nw(G)$.

Now suppose that $G$ is countable. Then $nw(G)=\omega$, so the equality
$w(G)=nw(G)$ implies that $G$ is second countable, in particular metrizable.
\end{proof}
\begin{remark}
\begin{itemize}
\item[(a)] The fact that every countable locally minimal group $(G,\tau)$  is metrizable
admits also a straightforward proof. Indeed,
let $\{x_n:\ n\in\N\}=G\setminus\{0\}$ and let $U_0$ be a \nbd\ of 0 such that
$G$ is $U_0$-locally minimal. Then there exists a sequence of symmetric neighborhoods
of zero $(U_n)$ satisfying $U_n+U_n\subseteq U_{n-1}$, $x_n\not \in U_n$,
$U_n \subseteq \bigcap_{k=1}^{n-1} (x_k +U_{n-1}-x_k)$ for all $n\in \N$. Since $\bigcap_{n=1}^\infty U_{n}=\{0\}$, the family $(U_n)$ forms
a base of \nbd s of 0 of a metrizable group topology $\sigma \leq \tau$ on
$G$ with $U_0\in \sigma$. Hence $\tau=\sigma$ is metrizable.

\item[(b)]  A similar direct proof shows that every locally minimal abelian
 group $(G,\tau)$ of  countable pseudocharacter is metrizable.
Here ``abelian" cannot be removed, since examples of minimal (necessarily
non-abelian) groups of countable pseudocharacter
and arbitrarily high character (in particular, non-metrizable) were built
by Shakhmatov (\cite{Sh}).
 \end{itemize}
\end{remark}
%%%%%%%%%%%%%%%%%%%%%%%%%%%%

%%%%%%%%%%%%%%%%%%%%%%%%%%%%
\subsection{Groups with no small (normal) subgroups}\label{subsection_NSS_groups}

In this subsection we show that groups with no small (normal) subgroups are
closely related to locally minimal groups and study some of their properties.

\begin{definition}\label{def_nsns} A topological group $(G,\tau)$ is
called {\em NSS group} (No Small Subgroups) if a suitable neighborhood
$V\in{\cal V}(0)$ contains only the trivial subgroup.

 A topological group $(G,\tau)$ is called {\em NSnS group} (No Small normal Subgroups) if a suitable neighborhood $V\in{\cal V}(0)$
 contains only the trivial normal subgroup.  \end{definition}

The distinction between NSS and NSnS will be necessary only when we
 consider non-abelian groups (or non-compact groups, see Remark \ref{cptUFSS} below).

\begin{example}\label{Zlinear} Examples for NSS and non-NSS groups.
\begin{itemize}
 \item[(a)] The unit circle $\mathbb{T}$ is a NSS group.
 \item[(b)] Montgomery and Zippin's solution to Hilbert's fifth problem asserts that every locally compact NSS group is a Lie group.
 \item[(c)] Any free abelian topological group on a metric space is a NSS group (\cite{MT}).
 \item[(d)] A dichotomy of Hausdorff group topologies on the integers: Any Hausdorff group topology $\tau$ on the integers is NSS if and only
if it is not linear. Indeed, suppose that $\tau$ is not NSS; let $U$ be a closed neighborhood of $0$. By assumption, $U$ contains a nontrivial
 closed subgroup $H$ which is of the form $n\Z$ ($n\ge 1$). Since $\Z/n\Z$ is a finite Hausdorff group, it is discrete and hence $n\Z$ is open in $\Z$. This shows that $\tau$ is linear.
\item[(e)] A group $G$ is {\em
topologically simple} if $G$ has no proper closed normal subgroups.
Every Hausdorff topologically simple group is NSnS. [Suppose that $G$ is topologically simple and Hausdorff  and let
$U\not=G$ be a closed neighborhood of $0$. Let $N$ be a normal
subgroup of $G$ contained in $U$. Then $\overline{N}$ is also a
closed subgroup of $G$ contained in $U$ and hence
$\{0\}=\overline{N}=N$. So $G$ is an NSnS group.
Actually a stronger property is true: if $G$ is Hausdorff and every closed normal
subgroup of $G$ is finite, then $G$ is NSnS (this provides a proof
of item (a)).] The infinite permutation group $G=S({\mathbb N})$ is
an example of a topologically simple group (\cite[7.1.2]{DPS}).

\end{itemize}
\end{example}

We omit the easy proof of the next lemma:

\begin{lemma}\label{NSS1}
\label{lem:NSS}
\begin{itemize}
\item[(a)]  The classes of NSnS groups and NSS groups are stable under taking finite direct products and finer group topologies.
\item[(b)]  The class of NSS groups is stable under taking subgroups.
\item[(c)] The class of NSnS groups is stable under taking dense subgroups.
%%%%%%%%%%%%%%%%%%%%%%%%%%%%

\item[(d)] No infinite product of  non-trivial groups is NSnS.
\end{itemize}
\end{lemma}

 Recall that a SIN group (SIN stands for Small Invariant Neighborhoods) is a topological group
 $G$ such that for every $U\in
{\cal V}(0)$ there exists $V\in {\cal V}(0)$ with $-x+V+x\subseteq U$ for all $x\in G.$

\begin{proposition}\label{loc_min+NSS>metr} Every locally minimal SIN group $G$ is metrizable
provided it is NSnS.
\end{proposition}

\begin{proof} Let us assume that $(G,\tau)$ is $V$--locally minimal and NSnS,
where $V$  is a neighborhood of $0$ in $(G,\tau)$  containing no
non--trivial normal subgroups. Since $\tau$ is a group topology, it is
possible to construct inductively a sequence $(V_n)$ of symmetric
neighborhoods of $0$ in $\tau$ which satisfy $V_n+V_n\subseteq V_{n-1}$
(where $V_0:=V$) and $-x+V_n+x\subseteq V_{n-1}$  for all $x\in G$.

Let $\sigma$ be the group topology generated by the neighborhood basis
$(V_n)_{n\in\N}$. Obviously, $\sigma$ is coarser than $\tau$
and $V\in {\cal V}_{\sigma}(0)$. In order to conclude that $\sigma=\tau$,
it only remains to show that $\sigma$ is a Hausdorff
topology, which is equivalent to $\dis \bigcap_{n\in\N}V_n=\{0\}$. This is
 trivial, since the intersection is a normal subgroup contained in $V$.
\end{proof}

\begin{example}\label{S(X)} One cannot relax the ``SIN" condition even when
$G$ is minimal. Indeed, for every infinite set $X$ the symmetric group
$G=S(X)$ is
minimal and NSnS.
On the other hand, $S(X)$ is metrizable only when $X$ is countable
(\cite[\S 7.1]{DPS}).
%$\spadesuit $\footnote{section7.1 ???}
Note that this group strongly fails to be NSS, as ${\mathcal V}(0)$ has a base consisting of
open subgroups (namely, the pointwise stabilizers of finite subsets of $X$).
\end{example}

\begin{remark}\label{Rem:NSS}
\begin{itemize}
\item[(a)]
%The classes of NSnS groups and NSS groups are stable under taking finite direct products and finer group topologies.
% \item[(b)]
 The completion of a NSS group is not NSS in general: For example
the group $\T^{\N}$ is monothetic, i.e. it has a dense
subgroup $H$ algebraically isomorphic to $\Z$. Since the completion of a
linear group topology is again linear, and the product topology on $\T^{\N}$
is not linear, $H$ is not linear  either. So Example \ref{Zlinear}(d)
implies
that $H$ is NSS. But $H$ is dense in $\T^{\N}$ which is not NSS.

\item[(b)] It was a problem of I. Kaplansky whether the NSS property is preserved under taking arbitrary quotients. A counter-example was
given by S. Morris (\cite{Mor}) and Protasov (\cite{Pts}); the latter proved that NSS is preserved under taking quotients with respect to discrete normal subgroups.

\item[(c)] In contrast with the NSS property, a subgroup of an NSnS group need not be NSnS. Indeed, take the permutation group $G=S({\mathbb N})$. Let
${\mathbb N}=\bigcup_n F_n$ be a partition of the naturals into finite sets $F_n$ such that each $F_n$ has size $2^n$. Let $\sigma_n$ be a
cyclic permutation of length $2^n$ of the finite set $F_n$ and let $\sigma$ be the permutation of ${\mathbb N}$ that acts on each $F_n$
as $\sigma_n$. Obviously, $\sigma$ is a non-torsion element of $G$, so it generates an infinite cyclic subgroup $C\cong \Z$. For convenience
identify $C$ with $\Z$. Then, while $G$ is NSnS by Example \ref{Zlinear}(e), the induced topology of $C$ coincides with the 2-adic topology of $C=\Z$, so it is linear and certainly non-NSnS. Indeed, a prebasic neighborhood of the identity element ${\rm id}_{\mathbb N}$
in  $C$ has the form  $U_x = C \cap {\rm Stab}_x$, where ${\rm Stab}_x$ is the stabilizer of the point  $x \in {\mathbb N}.$ If  $x\in F_n,$ then obviously all powers of $\sigma^{2^n}$ stabilize
$x$, so   $U_x$ contains the subgroup $V_n = \langle\sigma^{2^n} \rangle$. This proves that the induced topology of  $C \cong {\mathbb Z}$ is coarser than the 2-adic topology. Since
the latter is minimal (\cite[2.5.6]{DPS}), we conclude
that $C$ has the 2-adic topology.
\end{itemize}
\end{remark}

\section{Groups uniformly free from small subgroups}\label{section_ufss}

\subsection{Local minimality and the UFSS property }

 We have seen (Example \ref{Exa_LocMin}(c)) that all normed spaces are locally
 minimal when regarded as topological abelian groups. The
following group analog of a normed space was introduced by Enflo (\cite{E});
we will show in Facts \ref{UFFS_locmin}(a) that every such group is locally
minimal:

 \begin{definition} \label{defenflo}
 A  Hausdorff topological group $(G,\tau)$ is
 {\em uniformly free from small subgroups} (UFSS for short) if for some
 neighborhood $U$ of $0$, the sets $(1/n)U$ form a neighborhood basis at $0$ for $\tau$.
\end{definition}

Neighborhoods $U$ satisfying the condition described in Def. \ref{defenflo}
will be said to be {\em distinguished.} It is easy to see that any neighborhood
of zero contained in a distinguished one is distinguished, as well.

  Obviously, discrete groups are UFSS. Now we give some non-trivial examples.

\begin{example}\label{example1} \label{locallybounded}
\begin{itemize}
\item[(a)] ${\mathbb R}$ is a UFSS group with respect to $[-1,1].$ \item[(b)] ${\mathbb T}={\mathbb R}/{\mathbb Z}$ is a UFSS group with respect to ${\mathbb T}_+$, the image of $[-1/2,1/2]$ under the quotient map ${\mathbb R}\to {\mathbb T}$.
\item[(c)] A topological vector space is UFSS as a topological abelian group
if and only if it is locally bounded. In particular every normed space is a
UFSS group.

Recall that a subset $B$ of a (real or
complex) topological vector space $E$ is usually referred to as {\em bounded}
if for every neighborhood of zero $U$ in $E$ there exists
$\alpha>0$ with $B\subseteq \lambda U$ for every $\lambda$ with
$|\lambda|>\alpha,$ and {\em balanced} whenever $\lambda B\subseteq
B$ for every $\lambda$ with $|\lambda|\le 1.$ The space $E$ is
{\em locally bounded} if it has a bounded  neighborhood of zero. It is
straightforward that any locally bounded space is  UFSS when  regarded
as a topological abelian group, and any of its bounded neighborhoods
of zero is a distinguished neighborhood. Conversely, if a topological
vector
  space is UFSS as a topological abelian group,   then any distinguished
  balanced neighborhood of zero is bounded in this sense.

  This, of course, includes unit balls of normed spaces, but there are some
  important  non-locally-convex examples as well (see Example \ref{ExBana}(b)).
%  and \ref{lp}).
  \item[(d)] Every Banach-Lie group is UFSS, \cite[Theorem 2.7]{MP}.

  \end{itemize}
\end{example}

\begin{facts}
\label{UFFS_locmin} \begin{itemize} \item[(a)] Every UFSS group with
distinguished neighborhood $U$ is $U$-locally minimal {\rm
(\cite[Proposition 2.5]{MP})}. Indeed, one can see that a UFSS group $(G,\tau)$
with distinguished neighborhood $U$ has the following property, which trivially implies that $(G,\tau)$ is $U$-locally minimal: if
$\tcal$ is a group topology on $G$ such that $U$ is a $\tcal$-neighborhood  of $0$, then $\tau \leq \tcal$.
\item[(b)] All UFSS groups are NSS groups.
\end{itemize}
\end{facts}

Next we give some examples of NSS groups that are not UFSS.

\begin{example}\label{nss_not_ufss} \begin{itemize}
\item[(a)] Consider the group $\R^{(\N)}=\{(x_n)\in\R^{\N}:\ x_n =0 $ for almost all $n\in\N\},$ endowed with the rectangular topology, which admits as a basis of neighborhoods of zero the following family of sets:
$$
U_{(\varepsilon_n)}:=\{(x_n)\in \R^{(\N)}:\ |x_n|<\varepsilon_n \;\forall\ n\in\N \},\quad (\varepsilon_n)_{n\in {\mathbb N}}\in (0,\infty)^{\mathbb N}.
$$
 This group is not metrizable, hence it cannot be a UFSS group. On the other hand, any of the neighborhoods $U_{(\varepsilon_n)}$ contains only the trivial subgroup, so it is a NSS group.
\item[(b)] All free abelian topological groups on a metric space are NSS groups (see Example \ref{Zlinear}(c)).

Take a non-locally compact metric space $X$, then $A(X)$ is NSS, but not UFSS (indeed, if $A(X)$ is a $k$-space for some metrizable $X$, then $X$ is locally compact by
\cite[Proposition 2.8]{AOP}).
\end{itemize}
\end{example}

Example \ref{Exa_LocMin},  Facts \ref{UFFS_locmin} and Example \ref{example1} give   us a strong motivation to study locally
minimal groups, which put under the same umbrella three extremely relevant properties as minimality, UFSS and local compactness.

Example \ref{example1} shows that a locally minimal abelian group need not be precompact, in contrast with Prodanov-Stoyanov's
theorem. We see in the following example that actually there exist abelian locally minimal groups without nontrivial continuous characters.

\begin{example}\label{ExBana} \begin{itemize}
\item[(a)] According to a result of W. Banaszczyk (\cite{BB}), every infinite dimensional Banach space $E$ has a discrete  and free subgroup $H$ such that the quotient group $E/H$ admits only the trivial character. $E/H$ is locally isomorphic with $E$, hence it is a Banach-Lie group and then UFSS.

\item[(b)]
Fix any $s\in (0,1)$ and consider the topological vector space $L^{s}$ of all classes of Lebesgue measurable functions $f$ on $[0,1]$ (modulo almost everywhere equality) such
that $\int_0^1 |f|^s d\lambda$ is finite, with the topology given by the following basis of neighborhoods of zero:
$$
U_r=\left\{f\,:\,\int_0^1 |f|^s d\lambda \le r\right\},\quad r>0.
$$
(Following a customary abuse of notation, we use here (and in Example \ref{elecero} and Remark \ref{elecero_nolocmin}) the same symbol to denote both a function and its class under the equivalence relation of almost everywhere equality.)
In \cite{Day} it was proved that $L^s$ has no nontrivial continuous linear
functionals. It is known that every  character defined on the
topological abelian group underlying a topological vector space can be
lifted to a
continuous linear functional on the space
(\cite{Smi}). Thus as a topological group, $L^s$ has trivial dual. On the other
hand $L^{s}$ is a locally bounded space (note
that for every $r>0$ one has $U_1\subseteq r^{-1/s}U_r$), hence it is a UFSS
group (Example \ref{locallybounded}(c)).
\end{itemize}
\end{example}

\begin{remark}\label{cptUFSS}
It is a well known fact (see for instance \cite[32.1]{stroppel}) that for every compact group $K$ and $U \in {\mathcal V}(0)$ there exists a closed normal
subgroup $N$ of $K$ contained in $U$ such that $K/N$ is a Lie group, hence UFSS. This implies that the following assertions are equivalent:
\begin{itemize}
 \item[(a)] $K$ is UFSS,
 \item[(b)] $K$ is NSS,
  \item[(c)] $K$ is NSnS,
 \item[(d)] $K$ is a Lie group.
\end{itemize}
In case $K$ is abelian, they are equivalent to: $K$ is a closed subgroup of a
finite-dimensional torus.

 (The same equivalences are known to be true for locally compact groups which
 are either connected or abelian.)
\end{remark}

 In order to extend the above equivalences to locally minimal precompact groups,
 we need the following Lemma:

\begin{lemma}\label{LemmaNSS}
 Let $(G,\tau)$ be a precompact group. Then the following are equivalent:
\begin{itemize}
\item[(a)]  $(G,\tau)$  is NSnS;
\item[(b)]  For every $U \in {\mathcal V}(0)$ there exists a continuous
           injective homomorphism $f:G\to L$ such that  $L$ is a compact Lie
           group and $f(U)$ is a neighborhood of $0$ in $f(G)$.
\item[(c)] There exist a compact Lie group $L$ and a continuous injective
             homomorphism $f:G\to L.$
\item[(d)]  $G$ admits a coarser UFSS group topology.
\item[(e)]  $(G,\tau)$  is NSS.
\end{itemize}
In case $G$ is  abelian these conditions are equivalent to the existence of a
continuous injective homomorphism $G\to \T^k$ for
some $k\in \N$.
\end{lemma}
\begin{proof}
To prove that (a) implies (b) assume that $(G,\tau)$ is NSnS and fix a $U \in
{\mathcal V}(0)$. Let $W$ be a neighborhood  of $0$
in the completion $K$ of $G$ such that $W\cap G$ contains no non-trivial normal
subgroups and $(W+W)\cap G \subseteq U$. As in
Remark \ref{cptUFSS} there exists a closed normal subgroup $N$ of $K$ contained
in $W$ such that $L=K/N$ is a Lie group. As $N\cap
G=\{0\}$ by our choice of $W$, the canonical homomorphism $q: K\to L$ restricted
 to $G$ gives a continuous injective homomorphism $f=q\restriction_G :G \to L$.
 Observe that
$$
f(U)\supseteq q((W+W)\cap G)\supseteq q(N+W)\cap q(G)
$$
as $N \subseteq W$. Finally, the latter set  is a neighborhood of $0$ in $f(G)$
 as $N+W \in {\mathcal V}(0_K)$.

(b)$\Rightarrow$(c) is trivial. (c)$\Rightarrow$(d) is a consequence of the fact
 that every Lie group is UFSS. (d)$\Rightarrow$(e) and (e)$\Rightarrow$(a) are
 trivial.
%Let $(G,\tau)$ be a precompact group. Then the following are equivalent:
%\begin{itemize}
%\item[(a)]  $(G,\tau)$  is NSnS;
%\item[(b)]  for every $U \in {\mathcal V}(0)$ there exists a continuous injective homomorphism $f:G\to L$ such that  $L$ is a compact Lie group and $f(U)$ is a neighbourhood of $0$ in $f(G)$.
%\item[(c)]  $G$ admits a coarser UFSS group topology;
%\item[(d)]  $(G,\tau)$  is NSS.
%\end{itemize}
%In case $G$ is  abelian these conditions are equivalent to the existence of a continuous injective homomorphism $G\to \T^k$ for
%some $k\in \N$.
%\end{lemma}
%\begin{proof} \ To prove that (a) implies (b) assume that $(G,\tau)$ is NSnS and fix a $U \in {\mathcal V}(0)$. Let $W$ be a neighborhood  of $0$
%in the completion $K$ of $G$ such that $W\cap G$ contains no non-trivial normal subgroups and $(W+W)\cap G \subseteq U$. As in
%Remark \ref{cptUFSS} there exists a closed normal subgroup $N$ of $K$ contained in $W$ such that $L=K/N$ is a Lie group. As $N\cap
%G=\{0\}$ by our choice of $W$, the canonical homomorphism $q: K\to L$ restricted to $G$ gives a continuous injective homomorphism $f=q\restriction_G :G \to L$.
% Observe that
%$$
%f(U)\supseteq q((W+W)\cap G)\supseteq q(N+W)\cap q(G)
%$$
%as $N \subseteq W$. Finally, the latter set  is a neighbourhood of $0$ in $f(G)$ as $N+W \in {\mathcal V}(0_K)$.

%According to  Remark \ref{cptUFSS}, (b) $\Rightarrow$ (c). The implications (c) $\Rightarrow$ (d) and (d) $\Rightarrow$ (a) are clear.
\end{proof}

\begin{proposition}\label{Min+NSS} For a locally minimal precompact group $G$
the following are equivalent:
\begin{itemize}
\item[(a)] $G$ is NSnS;
 \item[(b)] $G$ is NSS;
\item[(c)] $G$ is UFSS;
\item[(d)] $G$ is isomorphic to a dense subgroup of a compact Lie group.
\end{itemize}
\end{proposition}
\begin{proof}
The implication (a) $\Rightarrow$ (d) follows from (a) $\Rightarrow $ (b) in
Lemma \ref{LemmaNSS}, since the local minimality of
$G$ and (b) from \ref{LemmaNSS} imply that $G\rightarrow L$ is an embedding.
Note that a compact subgroup of  a compact Lie group is
closed, so a Lie group itself. (d)$\Rightarrow$ (c) $\Rightarrow$ (b)
$\Rightarrow$ (a) are trivial.
\end{proof}

\begin{remark}\label{Min+NSS_abel} For locally minimal precompact abelian groups, condition (d) of Prop. \ref{Min+NSS} can be replaced by: $G$ is
isomorphic to a subgroup of a torus $\T^n$, $n\in \N$. Note that the class of locally minimal precompact abelian groups contains all minimal abelian
groups, due to the deep theorem of Prodanov and Stoyanov which states that such groups are precompact.
\end{remark}

\begin{remark} Proposition \ref{Min+NSS} shows very neatly the differences between minimality and UFSS. While all (dense) subgroups of a torus  $\T^n$
are UFSS,  the minimal among the dense subgroups of $\T^n$ are those that contain the socle $Soc(\T^n)$.

Indeed, $Soc(\T^n)$ is dense and every closed non-trivial subgroup $N$ of $\T^n$ is still a Lie group, so has non-trivial torsion elements (i.e.,
meets $Soc(\T^n)$). Therefore, by (\cite[Theorem 2.5.1]{DPS})  a dense
subgroup $H$ of $\T^n$ is   minimal  iff $H$ contains $Soc(\T^n)$. In particular, there is a smallest dense minimal subgroup of $\T^n$,
namely $Soc(\T^n)$.
\end{remark}

\begin{example} Let $\tau$ be a UFSS precompact topology on ${\mathbb Z}.$ Then $(\Z,\tau)$ is a dense subgroup of a group of the
form $\T^k\times \Z(m)$, where $k,m\in \N, k>0$. Indeed,  by Proposition \ref{Min+NSS} and Remark \ref{Min+NSS_abel} $(\Z,\tau)$ is isomorphic to
a subgroup of some finite-dimensional torus $\T^n$. Then the closure $C$ of $\Z$ in $\T^n$
will be a monothetic compact abelian Lie group. So the
connected component $c(C)\cong \T^k$ for some $k\in \N, k>0$ and $C/c(C)$
is a discrete monothetic compact group, so $C/c(C)\cong \Z(m)$ for some
$m\in \N$, so $C\cong \T^k\times \Z(m)$ since $c(C)$ splits as a divisible
subgroup of $C$.
\end{example}

\subsection{Permanence properties of UFSS groups} \label{perm_ufss}

 In the next proposition we collect all permanence properties of UFSS groups we can verify.

\begin{proposition}\label{perm_prop_3} The class of UFSS groups has the following permanence properties:
\begin{itemize}
\item[(a)]  If $G$ is a dense subgroup of $\wt{G}$ and $G$ is  UFSS, then
$\wt{G}$ is UFSS.
\item[(b)] Every subgroup of a UFSS group is UFSS.
\item[(c)] Every finite product of  UFSS groups is UFSS.
\item[(d)] Every group locally isomorphic  to a UFSS group  is UFSS.
\item[(e)]   If an abelian topological group $G$ has a closed subgroup $H$ such that  both $H$  and $G/H$ are UFSS,
then $G$ is UFSS as well.
\end{itemize}
\end{proposition}

\begin{proof} (a) Let $G$ be a  UFSS group with distinguished neighborhood $U$.
Note that closures in $\wt{G}$ of the neighborhoods of $0$
in $G$ form a basis of  the neighborhoods of $0$ in $\wt{G}$. Let ${W}$ be a
symmetric neighborhood of $0$ in $G$ which satisfies $G\cap (\ol{W}+ \ol{W})
\subseteq U$.

   Let us prove that
 $$
  (1/n)\ol{W} \subseteq \ol{(1/n)U }\quad\quad\forall\; n\in\N.
  $$
   To this end fix $x\in (1/n)\ol{W} $. This means $x,2x,\ldots,nx\in\ol{W}$.   Hence there exists a sequence $(x_k)$ in $W$ which tends to $x$ and the
   sequences $(jx_k)$ converge to $jx\in\ol{W}$ for $j\in\{1,\ldots,n\}$.   We may assume that $jx_k-jx\in \ol{W} $ for all $j\in \{1,\ldots,n\}$
   and all $k\in \N$, which implies $jx_k\in G\cap (\ol{W}+\ol{W})\subseteq U$
   for all $k\in\N$ and $j\in \{1,\ldots,n\}$. This implies $x_k\in  (1/n)U $ for all $k\in\N$ and hence
$x\in\ol{(1/n)U }$.

 The inclusion $(1/n)\ol{W} \subseteq \ol{(1/n)U }$ assures that  the sets $(1/n)\ol{W}$ form  a  neighborhood basis of $0$ in $\wt{G}$; i.e. $\ol{W}$
 is a distinguished neighborhood for $\wt{G}$.

(b) to (d) are easy to see.

(e) By assumption, there exists a neighborhood $W$ of $0$ in $G$ such that $\pi(W+W),$  and  $(W+W)\cap H,$ are
distinguished neighborhoods of zero in $G/H$ and $H$, respectively, where $\pi:G\to G/H$ denotes the canonical projection.

According to a result of Graev (\cite{G} or (5.38)(e) in \cite{HR}), $G$ is first countable, since  $H$ and $G/H$  have this property.

Let us show that
$$
 \forall (x_n)\; \text{with}\; x_n\in (1/n)W \Rightarrow x_n \stackrel{\tau}{\rightarrow} 0.\eqno(3)
$$
where $\tau$ is the original topology on $G$. Since $\pi((1/n)W) \subseteq (1/n)\pi(W)$, $\pi(W)$ is a distinguished neighborhood of zero
in $G/H$ and $G$ is first countable, there exists a sequence $(h_n)$ in $H$ such that $x_n-h_n\rightarrow 0$.

 For $n_0\in \N$, there exists $n_1\ge n_0$ such that for all $n\ge n_1$  we have
 $$
  h_n=x_n+(h_n-x_n)\in ((1/n)W+(1/n_0)W)\cap H   \subseteq ((1/n_0)W+
  (1/n_0)W)\cap H \subseteq (1/n_0)((W+W)\cap H).$$

Since the sets $(1/n)(((W+W)\cap H)$ form a basis of zero neighborhoods
in $H$, the sequence $(h_n)$ tends to $0$ and hence $(x_n)$ tends to $0$ as well.

Condition (3) implies that the family $((1/n)W)$ is a basis of zero neighborhoods for $G$. Indeed, fix $U\in {\cal V}_{\tau}(0)$ and suppose
$(1/n)W\not \subseteq U$ for every $n\in {\mathbb N}.$ Select $x_n\in (1/n)W, \; x_n\not \in U.$ According to (3) the sequence $(x_n)$
 converges to zero, which contradicts $x_n\not \in U\; \forall n \in {\mathbb N}.$
\end{proof}

\begin{remark}\label{RemUFSS0}
\begin{itemize}
\item[(a)]   Items (b) and  (c) imply that finite suprema of UFSS group
topologies
  are still UFSS. In the next section we will introduce the locally GTG
  topologies which, at least in the NSS case, can be characterized as
  arbitrary suprema of UFSS group topologies (see Definition \ref{defGTG}
   and Theorem \ref{Prop2GTG}).
\item[(b)] Item (c) follows also from (e). Let us note, that it cannot be strengthened to countably
infinite products. Indeed,  any infinite product of  non-indiscrete groups  (e.~g., copies of ${\mathbb T}$) fails to be NSS, so cannot be UFSS either.
\end{itemize}
\end{remark}

 The rest of the subsection is dedicated to a very natural property that was missing in Proposition \ref{perm_prop_3}, namely stability
under taking quotients and continuous homomorphic images. It follows from item (d) of this Proposition that a
quotient of a UFSS group with respect to a discrete subgroup is
UFSS. Actually it has been shown in \cite{MMP} (Proposition 4.5)
that every Hausdorff abelian UFSS group is a quotient group of a
subgroup of a Banach space. However, as we see in the next example,
a Hausdorff quotient of a UFSS group need not be UFSS.

\begin{example}
Let $\{e_n:\ n\in\N\}$ denote the canonical basis of the Hilbert space
$\ell^2$. Consider the closed subgroup $H:=\ol{\langle\{\frac{1}{n}e_n:\ n\in\N\}
\rangle}$ of $\ell^2$. Let us denote by $B$ the unit ball in $\ell^2$ and by $\pi:
\ell^2\rightarrow\ell^2/H$ the canonical projection.
For an arbitrary $\eps>0$, we will show that $\pi(\eps B)$ contains a nontrivial
subgroup. This will imply that the quotient $\ell^2/H$ is not NSS and, in
particular, is not UFSS.

Let $k_0\in\N$ such that $\dis \sum_{k>k_0}\frac{1}{k^2}<4\eps^2$. Let
$S$ be the linear hull of the set $\{e_k:\ k>k_0\}.$ We will obtain
$$
\pi(S)\subseteq \pi(\eps B).
$$
Indeed, fix $x=(x_n)\in S$. For $n>k_0$, there exists $k_n\in\Z$ such
that $\left |x_n-\frac{k_n}{n}\right |\le
\frac{1}{2n}$. Since $h:=\sum_{n> k_0} \frac{k_n}{n}e_n\in H$ and $\|x-h\|\le
\sqrt{ \sum_{n> k_0}(\frac{1}{2n})^2} < \eps$, we
obtain: $\pi(x)=\pi(h+(x-h))=\pi(x-h)\in\pi( \eps B)$ and hence $\dis \pi(S)
\subseteq \pi( \eps B)$.\end{example}

The next corollary shows that the class of {\em precompact} UFSS groups is
closed under taking arbitrary quotients.

\begin{corollary}\label{NEW_corollary}
If $G$ is a precompact UFSS group, then every continuous homomorphic image
of $G$ is UFSS.
\end{corollary}

\begin{proof} Let  $f: G \to G_1$ be a continuous surjective homomorphism. It can be extended to the respective compact completions
$f':\wt{G}\to \wt{G_1}$ of $G $ and $G_1$ respectively. Since $f$ is surjective and each group is dense in its completion, the
compactness of $\wt{G}$ yields that $f'$ is surjective. Moreover, $f'$ is open by the open mapping theorem. Hence $\wt{G_1}$ is isomorphic to a quotient of $\wt{G}$.  By Proposition \ref{perm_prop_3}(a) $\wt{G}$ is UFSS, hence (Remark \ref{cptUFSS}) $\wt{G}$ is a Lie group. Then $\wt{G_1}$
is a Lie group as well, so UFSS. This proves that $G_1$ is UFSS.
\end{proof}

\section{GTG sets and UFSS topologies}\label{Lydia_Chapter}
%%%%%%%%%%%%%%%%%%%%%%%%%%%%

%%%%%%%%%%%%%%%%%%%%%%%%%%%%

\subsection{General properties of GTG subsets}

%%%%%%%%%%%%%%%%%%%%%%%%%%%%

Vilenkin \cite{Vi} introduced locally quasi-convex groups while generalizing the
notion of a locally convex space. His definition is inspired on the description of closed symmetric subsets of vector spaces  given by the Hahn-Banach theorem.

Next we present a new generalization of locally convex spaces in the setting of topological groups which we will call {\em locally GTG groups}  where GTG abbreviates
{\bf g}roup {\bf t}opology  {\bf g}enerating (set). Similarly to the notion of a convex set (that depends only on the
linear structure of the topological vector space structure, but not on its topology), the
notion of a GTG set depends only on the algebraic structure
of the group.  In particular, it does not use any dual object at all,
 whereas the notion of quasi-convex set of a topological group $ G $ depends on the topology of $G$ via the continuity
of the characters to be used for the definition of the polar.

The  class  of locally GTG groups will be shown to contain all  locally quasi-convex groups, all locally
pseudoconvex spaces and all UFSS groups. As we will see, it fits very well in the setting of locally
minimal groups as it gives a connection between locally minimal groups and minimal groups (\ref{GTG+locMin}).
Moreover, we are not aware of any locally minimal group not having this property (see Question \ref{LocMin_vs_locGTG}).

Recall that a subset $A$ of a vector space  $E$ is called {\em pseudoconvex}
if $[0,1] A\subseteq A$ and $A+A\subseteq cA$ for suitable $c>0$. One may assume that $c\in \N$.
(Indeed, choose $\N\ni n>c$, then $cA\subseteq nA$ as $ca=(c/n)na\in[0,1] n A
\subseteq nA$ for all $a\in A$.)
Hence the set $A$ is pseudoconvex iff $[0,1] A\subseteq A$ and for some $n\in \N,\;  \frac{1}{n} A+\frac{1}{n} A\subseteq A$.
If $A$ is symmetric, this already implies that $(\frac{1}{n} A)$ forms a
neighborhood basis of a not necessarily Hausdorff group topology: $\frac{1}{nm}A+\frac{1}{nm}A=\frac{1}{m}(\frac{1}{n} A+\frac{1}{n} A)\subseteq \frac{1}{m} A$. A standard argument
shows that scalar multiplication is also continuous.

%Observe that for any symmetric convex subset $A$ of a
%vector space $V$ such that $0\in A$, the sets
%$\; \frac{1}{n} A=(1/n)A\;$ form, as $n$ runs over ${\mathbb N},$ a neighborhood basis
%at $0$ of a (not necessarily Hausdorff) group topology, which is the coarsest
%having $A$ as a neighborhood of zero.

It is well known that the unit balls of the  vector spaces $\ell^s$ where $0<s<1$ are pseudoconvex but not convex. The same can be said of their natural finite-dimensional counterparts  $\ell_n^s$, with $n\ge 2.$ Nevertheless, by far not all symmetric subsets of a vector space are pseudoconvex, as we see in the next example.

\begin{example}
\label{notgtgset} The subsets of  ${\mathbb R}^2$: $U=([-1,1]\times\{0\}) \cup (\{0\}\times [-1,1])$ and $V=(\mathbb R\times\{0\}) \cup (\{0\}\times \mathbb R)$ are symmetric and not pseudoconvex. Observe that $[0,1]U\subseteq U$; $\frac{1}{n}U=([-1/n,1/n]\times\{0\}) \cup (\{0\}\times [-1/n,1/n])$;  $[0,1]V\subseteq V$ and $\frac{1}{n}V=V$.
\end{example}

\begin{definition}\label{definicion_gtg}

Let $G$ be  an abelian group and let $U$ be a symmetric subset of $G$ such
that $0\in U.$ We say that $U$ is a {\em group topology
generating} subset of $G$ (``GTG subset of $G$" for short) if the sequence
of subsets $\{(1/n)U \,:\,n\in \N\}$ is a basis of
neighborhoods of zero for a (not necessarily Hausdorff) group topology ${\mathcal T}_U$ on $G$.
\end{definition}

In case $U$ is a GTG set in $G$, ${\cal T}_U$ is the coarsest group topology on $G$ such that
$U$ is  a neighborhood.

We do not know whether the following natural converse is true:  Let $G$ be an abelian group and $U$ a symmetric subset of $G$ which contains zero and such that there
 exists the coarsest group topology on $G$ for which $U$ is a neighborhood of zero. Then $U$ is a GTG set.

\begin{example}\label{ejemplosgtg}
\begin{itemize}
\item[(a)] Every symmetric distinguished neighborhood of zero in a UFSS group is a GTG set.
\item[(b)] Every subgroup of a group $G$ is a GTG subset of $G$.
\end{itemize}
\end{example}

\begin{proposition} \label{definicion_bis_gtg} A symmetric subset $U\subseteq G$ of an abelian group $G$
is a GTG subset if and only if
  $$
  \exists m\in \N\ \mbox{with} \ (1/m)U + (1/m)U \subseteq U .\eqno(*)
  $$

Moreover, if $U$ is a GTG set, $U_{\infty}=\bigcap_{n=1}^{\infty}(1/n)U$ is the ${\mathcal
T}_U$-closure of $\{0\}$ and in particular, it is a closed subgroup and a $G_{\delta}$ subset of $(G,{\mathcal T}_U)$.
\end{proposition}

\begin{proof} The given condition  is obviously necessary. Conversely, to prove that addition is
continuous, we are going to see that $(1/mn)U
+ (1/mn)U \subseteq (1/n)U\quad\forall n\in\N .$ Fix $x,\ y\in (1/mn)U$ and observe that $jx,\ jy\in (1/m)U$
for all $1\le j\le n$. This implies $j(x+y)=jx+jy\in (1/m)U+ (1/m)U \subseteq U $ for all $1\le j\le n$ and
hence $x+y\in (1/n)U$.

If $U$ is a GTG set, then ${\mathcal T}_U$ is a group topology of $G$, hence $U_\infty=\overline{\{0\}}^{ {\mathcal T}_U}$ is a subgroup of $G$.
\end{proof}

Proposition \ref{definicion_bis_gtg} gives the possibility to
 define a GTG set in a more precise way. Namely, one can introduce the following
 invariant for a symmetric subset $U\subseteq G$
of an abelian group $G$ with $0\in U$
$$
\gamma(U) := \min \{m\in \N: (1/m)U + (1/m)U \subseteq U\}%, \eqno(\dag)
$$
with the usual convention $\gamma(U) = \infty$ when no such $m$ exists. According to Proposition \ref{definicion_bis_gtg}, $U$ is a GTG set iff
$\gamma(U) <\infty$. Let us call $\gamma(U)$ the {\em GTG-degree} of $U$, it obviously measures the GTG-ness of the symmetric set
$U$ containing 0. Clearly, $U$ has GTG-degree 1 precisely when $U$ is a subgroup.
(Compare this with the {\em modulus of concavity} defined in \cite[3.1]{Rol}.)

\begin{proposition}\label{pcGTG} A symmetric subset $A$ of a vector space $E$ which satisfies  $[0,1]A\subseteq A$ is GTG iff it is pseudoconvex.
 \end{proposition}

\begin{proof} By assumption, $[0,1]A\subseteq A$. This implies that $(1/n)A=\frac{1}{n}A$.
 In the introduction to this section we have already shown that a
symmetric set $A$ is pseudoconvex
if and only if it satisfies $\frac{1}{n}A+\frac{1}{n}A\subseteq A$ for some $n\in \N$. Since
$(1/n)A=\frac{1}{n}A,$ it is a consequence of  \ref{definicion_bis_gtg} that $A$ is pseudoconvex iff it is GTG.
\end{proof}

\begin{example}
 The subsets of $U=([-1,1]\times\{0\}) \cup (\{0\}\times [-1,1])$ and $V=(\mathbb R\times\{0\}) \cup (\{0\}\times \mathbb R)$ of Example \ref{notgtgset} are not GTG sets.
 \end{example}

\begin{remark}
Let  $U$ be a symmetric subset of an abelian group $G$ with $0\in U$.
We analyze the behaviour of the sequence $(1/n)U$ in the following cases of interest:
\begin{itemize}
\item[(a)]  If $U_\infty=\{0\}$, then $U$ is  a GTG set iff $(G,{\cal T}_U)$ is UFSS.
%\footnote{I eliminated the sequence "(observe the set $U$ of Example \ref{notgtgset})."}

\item[(b)]  Now assume that $(1/m)U = U_\infty$ for some $m$. Then $U$ is GTG iff $U_\infty $
is a subgroup. It is clear that $(1/m)U = U_\infty$ is a union of
cyclic subgroups.

We know (Proposition \ref{definicion_bis_gtg}) that if $U$ is a GTG set, then $U_\infty$ must be a subgroup.
 But in this circumstance, we can invert the implication. Indeed, if
$U_\infty=(1/m)U$ is a subgroup, then obviously $(1/m)U+(1/m)U\subseteq (1/m)U\subseteq U$ holds true, so that
$U$ is a GTG set.
% \footnote{I eliminate", by Remark \ref{kgtg}." This does not help here.}
This fact explains once more why the subset $V=V_\infty$ from Example \ref{notgtgset} is not a GTG set
(simply it is not a
subgroup).
  \end{itemize}
  (Note that  we are not considering here the third possibility: $U_\infty\ne \{0\}$ yet the chain $(1/m)U$ does not stabilize.)
\end{remark}

\begin{remark}\label{kgtg}
Let $U$ be a symmetric subset of  a group $G$. Then the following holds true:
\begin{enumerate}
\item[(a)] $(1/n)((1/m)U)=(1/m)((1/n)U)$ for all $n,m\in \N$.
\item[(b)] For symmetric subsets $A$ and $B$ of $G$ and $k\in \N$ we have:
$(1/k)A+(1/k)B\subseteq (1/k)(A+B)$.
\item[(c)] The following assertions are equivalent:
\begin{enumerate}
\item[(i)] $U$ is a GTG set in $G$.
\item[(ii)] For every $k\in\N$ the set $(1/k)U$ is a GTG set in $G$ .
\item[(iii)] There exists $k\in \N$ such that $(1/k)U$ is a GTG set in $G$.
\end{enumerate}
In this case ${\cal T}_U={\cal T}_{(1/k)U}$ for every $k\in\N$.
\end{enumerate}
\end{remark}

\begin{proof}
(a) and (b) are straightforward.

(c) (i) $\Rightarrow$ (ii): Suppose that $(1/m)U+(1/m)U\subseteq U$. This yields
$(1/m)(1/k)U+(1/m)(1/k)U\stackrel{(a)}{=}(1/k)(1/m)U+(1/k)(1/m)U
\stackrel{(b)}{\subseteq}(1/k)[(1/m)U+(1/m)U]\subseteq (1/k)U$ and hence the assertion
follows from Proposition \ref{definicion_bis_gtg}.

(ii) $\Rightarrow$ (iii) is trivial.

(iii) $\Rightarrow$ (i) Let $m$ be such that $(1/m)((1/k)U)+(1/m)((1/k)U)\subseteq (1/k)U$.
Since $(1/mk)U\subseteq (1/m)((1/k)U)$ we deduce
$$ (1/mk)U+(1/mk)U\subseteq (1/k)U\subseteq U$$
and the assertion is a consequence of Proposition \ref{definicion_bis_gtg}.

Finally, assume that $U$ is a GTG set. From $(1/mk)U\subseteq  (1/m)((1/k)U)\subseteq (1/m)U$, we obtain
the equality of the topologies ${\cal T}_U={\cal T}_{(1/k)U}$.
\end{proof}

Next we give  investigate under which conditions intersections and products of GTG sets are GTG.

\begin{lemma}\label{Rem_x}\label{inv_im_GTG}
\begin{itemize}
\item[(a)] Inverse images of GTG sets by group homomorphisms are GTG.
More precisely, if $\phi: G \to H$ is a homomorphism and $A \ni 0$ is a symmetric subset of $H$, then
$\gamma(\phi^{-1}(A))\leq \gamma(A)$. If $A\subseteq \phi(G)$ then $\gamma(\phi^{-1}(A))=\gamma(A)$.
 \item[(b)] If  $\{A_i: i\in I\}$ is a  family of GTG sets of a group $G$ and the subset $\{\gamma(A_i): i\in I\}$ of $\N$ is bounded,
then also $\bigcap_{i\in I} A_i$ is a GTG subset of $G$. In particular, the
intersection of any finite family of GTG sets of $G$ is a GTG set
of $G$.
\item[(c)]
Let $(G_i)_{i\in I}$ be a family of groups and let $A_i$ be a subset of $G_i$  for every $i\in I$. The set $A:= \prod_{i\in
I}A_i\subseteq \prod_{i\in I} G_i$ is a GTG set of $G:=\prod_{i\in I} G_i$ iff all $A_i$ are GTG sets and the subset
$\{\gamma(A_i): i\in I\}$
of $\N$ is bounded.
In particular,
\begin{itemize}
   \item[(c$_1$)] if $I$ is finite then $\prod_{i\in I}A_i$ is GTG iff all the sets $A_i$ are GTG.
   \item[(c$_2$)] for an arbitrary index set $I$,
     $U$ is a GTG set of a group $G$ iff  $U^I$ is a GTG set of $G^I$.
  \end{itemize}
%    Also, if all $A_i$ are GTG subsets and almost all of them are the whole
%    group then
%    $\prod_{i\in I}A_i$ is a GTG subset of $\prod_{i\in I}G_i $.
  \end{itemize}
 \end{lemma}
\begin{proof}
(a) is a consequence of the identity $\; (1/m)\phi^{-1}(A)=\phi^{-1}((1/m)A).$

(b)  It is straightforward to prove that $(1/m)\bigcap_{i\in I}A_i=\bigcap_{i\in I}(1/m) A_i$.
By our hypothesis we may choose $m$ so large that $(1/m)A_i+(1/m)A_i\subseteq A_i$
for all $i\in I$ and obtain $(1/m)\bigcap_{i\in I}A_i+(1/m)\bigcap_{i\in I}A_i\subseteq  \bigcap_{i\in I}A_i.$
The assertion follows from Proposition \ref{definicion_bis_gtg}.

  (c) follows easily from (a), Proposition \ref{definicion_bis_gtg} and the equality
$(1/n)\prod_{i\in I}A_i=\prod_{i\in  I}(1/n)A_i$.
\end{proof}

\begin{example}\label{nonGTGex}
\begin{itemize}
\item[(a)]
Let ${\mathbb P}$ be the set of all positive primes. For each $p\in {\mathbb P}$ we define the symmetric subset of ${\mathbb Z}$
$$
U_p=\{0\}\cup \{\pm 2^{n_2}3^{n_3}\cdots p^{n_p}: n_2,\,n_3,\cdots , n_p\in {\mathbb N}\cup \{0\}\}.
$$
Note that for $p,\,q\in {\mathbb P},$ we have $(1/q)U_p=U_p$ for $q\le p$ and $(1/q)U_p=\{0\}$ otherwise. This implies that for every
$p\in {\mathbb P},$ $ U_p$ is a GTG set, $(U_p)_{\infty}=\{0\}$ and $U_p+U_p\not\subseteq U_p.$ Hence $p< \gamma(U_p)$.
The subset $U=\prod_{p\in{\mathbb P}}U_p\subseteq {\mathbb Z}^{\mathbb P}$ is symmetric and satisfies $U_{\infty}=\prod_{p\in\P}(U_p)_\infty= \{0\}$, but it is not a GTG set by Lemma \ref{Rem_x}(c).

  Define $V_p:=U_p\times \prod_{q\in \P,\, q\not= p}\Z$. Then for every  $p\in \P$ the sets $V_p$ are GTG, however, their intersection $\bigcap_{p\in\P}V_p=U$ is not GTG as shown above.
\item[(b)]
 A simpler example of a non-GTG intersection of GTG sets can be obtained from the set $U$ of Example \ref{notgtgset}: it is the intersection of all $||\cdot||_{1/n}$-unit balls $U_n$
 in ${\mathbb R}^2$, for $n\in \N.$
\item[(c)]  If $ U_n$  is the subset of $ G_n = \R^2$, as in (b),
%$the $\|.\|_{1/n}$-unit ball in $ G_n = \R^2$,
 then $\gamma(U_n) \to +\infty$. Therefore, $U = \prod_{n\in \N } U_n$ is not a GTG set in $G= (\R^2)^\N$, according to item (c) of Lemma \ref{Rem_x}.
 \end{itemize}
\end{example}

%%%%%%%%%%%%%%%%%%%%%%%%%%%%\NB\begin{example}\label{ex_nogtg} \NB \end{example}

The next proposition  give an intuitive idea about GTG sets:

\begin{proposition}\label{ex_nogtg}
  If $G$ is a compact connected abelian group and $U$ is a GTG set of $G$ with Haar measure 1, then $U=G.$
\end{proposition}

\begin{proof} For every positive $n$ the map $f_n:G \to G$ defined by $f_n(x)=nx$ is a surjective continuous endomorphism
(such a group $G$ is always divisible, see e. g. \cite[24.25]{HR}). Since
every surjective continuous endomorphism is measure preserving (\cite{Halmos}), one has $\mu(f_n^{-1}(U))=\mu(U)=1$. Therefore, also
$$
U_\infty = \bigcap_n f_n^{-1}(U)
$$
has measure 1. Since  $U_\infty$ is a subgroup, this is possible only when $U_\infty=G$. This yields $U=G.$
\end{proof}

\subsection{Construction of GTG sets and UFSS topologies}

Now we shall propose a general construction for building infinite GTG sets in abelian groups.
In case the group is complete metric, the GTG set can be chosen compact and totally disconnected.

%In the next remark we fix the notation need in the construction.

\begin{remark}\label{remark:NIndep} In the construction we shall need the following sets of sequences of integers:
$$
\fZ=\Z^{\N_0},  \  \   K_m= \left\{(k_j)\in \fZ: \sum_{j=0}^\infty \frac{|k_j|}{2^j}\le \frac{1}{2^m}\right\}
 \ \ \mbox{ for   $\ m\in \Z\ $, and } \  \  \fP = \prod_{j=0}^\infty \{0, \pm 1, \pm 2, \pm 3, \ldots, \pm 2^{j+2}\}
$$
\begin{itemize}
\item[(a)]  Obviously, $K_m \subseteq \fP$ when $m\geq -2$, and  $K_m + K_m \subseteq  K_{m-1}$, for $m\in \Z$.
\item[(b)] We use also the direct sum $\fZ_0=\bigoplus _{\N_0}\Z$.
% and the sets
%$$
%K_{m,0}:=K_m \cap \fZ_0= \left\{(k_j)\in \fZ_0: \sum_{j=0}^\infty\frac{|k_j|}{2^j}\le  \frac{1}{2^m}\right\}  \ \ \mbox{ and } \  \  \fP_0 = \fP \cap \fZ_0.
%$$
For $(a_n) \in  \fZ_0$ and any sequence $(x_n)$ of elements of $G$ the sum $\sum_{j=0}^\infty a_jx_j$
makes sense and will be used in the sequel.  In this way, every element  ${\mathbf x}= (x_n)\in G^ {\N_0}$
gives rise to a group homomorphism  $\varphi_{\mathbf x}: \fZ_0\longrightarrow G$
defined by $\varphi_{\mathbf x}((a_n)) := \sum_{j=0}^\infty a_jx_j$ for $(a_n) \in  \fZ_0$.
\item[(c)] $\fZ$ will be equipped with the product topology, where $\Z$ has the discrete topology with
basic open neighborhoods of 0 the subgroups
$$
W_n= \{(k_j)\in \fZ: k_0=k_1=\ldots=k_n=0\},
$$
$n\in \N_0.$
Thus, $\fP$ is a compact zero-dimensional subspace of $\fZ$. Let us see that $K_m$ is closed in $\fP$
for $m\in \Z$, hence a compact zero-dimensional space
on its own account. Indeed, pick $\xi= (k_j)_{j\ge 0}\in \fP \setminus K_m$. Then
$ \sum_{j\ge 0}\frac{|k_j|}{2^j}>  \frac{1}{2^m}$, so $ \sum_{j=0}^n
\frac{|k_j|}{2^j}>  \frac{1}{2^m}$ for some index $n$.  Hence the neighborhood  $(\xi + W_n)\cap \fP$
%\prod_{j=0}^{j_0}\{k_j\}\times \prod_{j>j_0}\{0, \pm 1, \pm 2, \pm 3, \ldots, \pm 2^{j}\}$
of $\xi$ misses the set $K_m$.
\end{itemize}
\end{remark}

A sequence $(x_n)_{n\ge 0}$ in $G$ will be called {\em nearly independent},  if  it satisfies
$$
\sum_{j=0}^\infty a_jx_j = 0\quad \Longrightarrow  \quad  (a_n) =0 \;\;\;\;  \mbox{ for all } \;\; \;\;  (a_n)\in {\mathfrak P}\cap \fZ_0.
\eqno(4)
$$
%i.e., if $\ker \varphi_{\mathbf x} \cap {\mathfrak P}_0 = 0$.
This term is motivated by the fact, that usually a sequence $(x_n)_{n\ge 0}$ in $G$ is called
independent, if $\ker \varphi_{\mathbf x} = 0$.

\begin{claim}\label{Claim1}
If $G$ is an abelian group and ${\mathbf x}= (x_n)$ is a nearly independent sequence of $G$, then $\varphi_{{\mathbf x}} \restriction_{K_{-1}\cap \fZ_0 }: K_{-1}\cap \fZ_0 \to G$
is injective.
\end{claim}

\begin{proof}
Assume that $\varphi_{\mathbf x}((k_j))=\varphi_{\mathbf x}((l_j))$ for $(k_j)\in K_{-1}$ and $(l_j)\in K_{-1}$.
Then $0=\sum_{j=0}^n(k_j-l_j)x_j$ and $|k_j-l_j|\le 2^{j+2}$, combined with near independence,  imply $k_j=l_j$ for all $j$.
\end{proof}

The following lemma reveals a sufficient condition under which an abelian group $G$ admits a non-discrete UFSS group topology, namely the existence of a nearly independent sequence. The necessity of this condition will be established at a later stage (see Corollary \ref{LastCorollary}).

\begin{lemma}\label{LemmaUFSS}
Let $G$ be an abelian group and let ${\mathbf x}= (x_n)$ be a
nearly independent sequence of $G$. Then  the set
$
X:= \varphi_{\mathbf x}(K_0 \cap \fZ_0)
%=\left\{\sum_{j=0}^nk_jx_j:\ n\in\N,\ k_j\in\Z,\ \sum_{j=0}^n\frac{|k_j|}{2^j}\le 1\right\}
$
is a GTG subset of $G$ with $\gamma(X)=2$. More precisely,
$$
(1/2^m)X=\varphi_{\mathbf x}(K_m \cap \fZ_0)= \left\{\sum_{j=0}^nk_jx_j:\ n\in\N,\ k_j\in\Z,\ \sum_{j=0}^n\frac{|k_j|}{2^j}\le \frac{1 }{2^m}\right\},\eqno(5)
$$
$X_\infty =0$ and $(x_n)$ tends to $0$ in ${\cal T}_X$, so ${\cal T}_X$ is a non-discrete UFSS topology.
\end{lemma}

\begin{proof} The inclusion $\supseteq$ in (5) is obvious. We prove the following stronger version of the reverse inclusion by induction:
$$
\mbox{if }\; x=\sum_{j=0}^nk_jx_j \in X\mbox{ with} \; (k_j) \in K_0,  \mbox{ then } x\in (1/2^m)X \;\; \Longrightarrow \;\; (k_j) \in K_m\eqno(6)
$$
For $m=0$ the assertion is trivial. So suppose (6) holds true for $m$ and let $x=\sum_{j=0}^nk_jx_j \in (1/2^{m+1})X$,  with
$(k_j) \in K_0$. Since $x,2x\in (1/2^m)X$,  our the induction hypothesis gives $(k_j) \in K_m$.
% i.e.,  $\sum_{j=0}^n\frac{|k_j|}{2^j}\le \frac{1}{2^m}$.
 Moreover, there exists a  representation $2x=\sum_{j=0}^n l_jx_j$ with $(l_j) \in K_m$, i.e.,
$\sum_{j=0}^n\frac{|l_j|}{2^j}\le \frac{1}{2^m}$. (Observe that without loss of generality we may assume that the upper index for the summation may be assumed to be equal for $x$ and $2x$.) Then $\varphi_{\mathbf x}((2k_j)) = \varphi_{\mathbf x}((l_j))$ with $(2k_j), (l_j) \in K_{-1}$, so Claim \ref{Claim1} applies $2k_j=l_j$ for all $j$ and hence $\sum_{j=0}^n\frac{|k_j|}{2^j}= \sum_{j=0}^n\frac{|l_j|}{2^{j+1}} \le \frac{1}{2^{m+1}}.$ This proves (6), and consequently also (5). Obviously, (6) yields also $X_\infty=\{0\}$.

 For $m=1$ the equation (5) and $K_1 + K_1 \subseteq K_0$ give  $(1/2)X+(1/2)X\subseteq X$. Hence $\gamma(X)\le 2$, and consequently $X$ is a GTG set and ${\cal T}_X$ is a UFSS topology.

For a fixed $N \in \N$ the definition of $X$ and (5) give $x_n \in (1/2^N)X$ for all $n\geq N$, so $\{x_n:\ n\ge N\}\subseteq(1/2^N)X $. This shows that $x_n\to 0$ in ${\cal T}_X$ and so ${\cal T}_X$ is not discrete.

Finally, to prove that $\gamma(X)\ge 2$ it  suffices to observe that $\gamma(X)= 1$ would imply that $X$ is a subgroup, so $X=X_\infty$. Now $X_\infty=\{0\}$ contradicts the non-discreteness of  ${\cal T}_X$.
\end{proof}

Let $(G,d)$ be a metric abelian group, let $v$ be the group seminorm associated to the  metric
$d$ (i.e., $v(x) = d(x,0)$ for $x\in G$) and let $B _\eps=\{x\in G:\ v(x)\le\eps\}$ be the closed
disk with radius $\eps$ around 0. For a nearly independent sequence  $(x_n)$ of $G$ and a non-negative $n\in \Z$ let
$$
\eps_n :=\min\left\{v\left(\sum_{j=0}^{n}a_j x_j\right):\ |a_j|\le 2^{j+2}, \ (a_j)\not=(0)\right\}>0. \eqno(7)
$$
We call the sequence $(x_n)$ {\em almost independent}, if the inequality
$$
 2^{n+3}v(x_{n+1})<\eps_n\le v(x_{n}) \eqno(8)
$$
 holds. Note that $\eps_n\le v(x_{n})$ obviously follows from the definition of $ \eps_n$.
Moreover, every almost independent sequence  (rapidly) converges to 0 in $(G,d)$.

It is straight forward to prove that a subsequence of a strictly, respectively almost independent
sequence is again strictly, respectively almost independent.
The motivation to introduce the sharper notion of almost independent sequence is given in the
lemma below. First we need to isolate a property that will be frequently used in the sequel:

\begin{claim}\label{Claim2} If  $(G,d)$ is a metric group and $(x_n)$ is an almost independent sequence of $G$, then $\varphi_{\mathbf x}(K_{m}\cap W_n \cap \fZ_0)\subseteq B_{\frac{v(x_n)}{2^{m+2}}}$ for any $m \in \Z$ and $n\geq 0$.
\end{claim}

\begin{proof} We have to prove that
$
v\left(\sum_{j=n+1}^kk_jx_j\right) <  \frac{1}{2^{m+2}}v(x_n),   \mbox{ whenever } \  (k_j) \in K_m.
% \ \mbox{ i.e.,  } \ \sum_{j=0}^n\frac{|k_j|}{2^j}\le \frac{1 }{2^m}. \eqno(10)
$
This follows applying (8) to the term $2^j v(x_j)$ in
%$j_0\leq k$ and $x=y_k-y_{j_0}=\sum_{j=j_0+1}^kk_jx_j \in G$ one has  due to the following
%estimates: for $X \ni x=\sum_{j>j_0}k_jx_j$ with $\sum_{j>j_0}\frac{|k_j|}{2^j}\le 1$
$$
v\left(\sum_{j=n+1}^kk_jx_j\right)\le \sum_{j=n+1}^k |k_j| v(x_j) = \sum_{j=n+1}^k   \frac{|k_j|}{2^j} {2^j} v(x_j)<  \sum_{j=n+1}^k  \frac{|k_j|}{2^j} \frac{1}{4}v(x_{j-1})
\leq \frac{1}{4}v(x_{n}) \sum_{j=n+1}^k  \frac{|k_j|}{2^j}\leq  \frac{1}{2^{m+2}}v(x_{n}).
 $$
\end{proof}

\begin{lemma}\label{Lydia_Lemma} Let $(G,d)$ be a metric group and let  $(x_n)$ be an almost independent sequence of $G$. Then
\begin{itemize}
\item[(a)]  the non-discrete UFSS topology ${\mathcal T}_{X}$ generated by the  GTG set $X$ of  $G$
corresponding to $(x_n)$ as in Lemma \ref{LemmaUFSS}, is
 finer than the original topology of $G$;
\item[(b)] the subsequence $(x_{2n})$ is still almost independent, and for the GTG set $Y$ of  $G$ corresponding to $(x_{2n})$ as in Lemma \ref{LemmaUFSS},
   ${\mathcal T}_{X}< {\mathcal T}_{Y}$.
\end{itemize}
\end{lemma}

\begin{proof} (a) We have to prove that for a given $\eps>0$, there exists $m\in \N$ such that $(1/2^{m})X\subseteq B_\eps$.
 Since $x_n\to 0$ in the metric topology, there exists $m\in \N$ such that $\frac{v( x_{m-1}) }{2^{m+2}}<\eps$.
As $K_{m}\subseteq W_{m-1}$ and $ (1/2^{m})X= \varphi_{\mathbf x}(K_{m}\cap \fZ_0)$, from
Claim \ref{Claim2} we obtain  $(1/2^{m})X= \varphi_{\mathbf x}(K_{m}\cap \fZ_0)=\varphi_{\mathbf x}(K_{m}\cap W_{m-1}\cap \fZ_0)\subseteq  B_\eps$.

(b) By Lemma \ref{LemmaUFSS}, ${\cal T}_Y$ is a UFSS  topology on $G$. Since $Y\subseteq X$, we trivially have ${\cal T}_Y\supseteq {\cal T}_X$.
It remains to be shown that ${\cal T}_Y$ is strictly finer than ${\cal T}_X$.  Let us prove that the ${\cal T}_X$ null--sequence $(x_{2n+1})$ does not converge to $0$ in ${\cal T}_Y$. It is enough to show that $\{x_{2n+1}:\ n\in\N\}\cap Y=\emptyset$. So assume $x_{2m+1}=\sum_{j=0}^nk_jx_{2j}$ for some $m\in\N$ and $(k_j)\in K_0$.  As $x_{2m+1} \in \varphi_{\mathbf x}(K_0)$ as well, this contradicts Claim \ref{Claim1}.
\end{proof}

In the next theorem we show that the set $X$ from the previous lemmas,
corresponding to an almost independent sequence of $G$, has a compact totally disconnected closure when $G$ is complete.

\begin{theorem}\label{Lydia1} Let $(G,d)$ be a complete metric group and let  $(x_n)$ be an
almost independent sequence of $G$. Then the closure   $\wt{X}$ of the GTG set $X$  corresponding
to $(x_n)$ as in Lemma \ref{LemmaUFSS}, is compact and totally disconnected.
Moreover,  $\wt{X}$ is a GTG set with $\gamma(\wt{X})=2$, so ${\cal T}_{\wt{X}}$ is a non-discrete
UFSS topology finer than the original topology of $G$.
\end{theorem}

\begin{proof} We intend to extend the map $  \varphi_{\mathbf x}$ defined in item (b) of Remark
\ref{remark:NIndep} to a map $  \varphi: \bigcup_{m\in \Z }K_m \longrightarrow G$ by setting
$ \varphi( (k_j)) =\sum_{j\ge 0}k_jx_j $ (the correctness of this definition is checked below).
Furthermore, we show that $\varphi\restriction_{K_m}$ is continuous for each $m$, while
$\varphi\restriction_{K_0}$ is injective. Since each $K_m$ is a compact zero-dimensional space
(Remark \ref{remark:NIndep} (c)), this will prove that $\wt{X}=\varphi(K_0)$ itself is a compact
zero-dimensional space, while the subspaces $\varphi(K_m)$ with $m< 0$ are just compact.

For a fixed $(k_j)_{j\ge 0}\in K_m$ let $y_ n = \sum_{j=0}^nk_jx_j \in G$. To see that $(y_n)$ is a
Cauchy sequence in $G$ apply
%for $j_0\leq k$ and  $x=y_k-y_{j_0}=\sum_{j=j_0+1}^kk_jx_j \in G$ one has $v(x) <  \frac{1}{2^{m+2}}v(x_{j_0})$ due to
  Claim \ref{Claim2} to get $v(y_k - y_n)\leq \frac{1}{2^{m+2}} v(x_{n})$ for every pair $n\leq k$. Since $x_n \to 0$, this proves that $(y_n)$ is a   Cauchy sequence in $G$.
Since $G$ is complete, the litmit $\lim y_n$ exists and $ \varphi( (k_j)) =\sum_{j\ge 0}k_jx_j$ and $\wt{X}:=\varphi(K_0) $  make sense.
%$$\wt{X}:=\varphi(K_0) = \left\{\sum_{j\ge 0}k_jx_j:\ \sum_{j\ge 0}\frac{|k_j|}{2^j}\le 1\right\}.$$

Since the norm function $v: G \to \R$ is continuous, we obtain from Claim \ref{Claim2}, after passing to the limit:

$$
\varphi(K_m \cap W_{n}) \subseteq B_{\frac{v(x_{n})}{2^{m+2}}}
% \;\;\;\; (i.e., \;\;   v\left(\sum_{j=j_0+1}^\infty k_jx_j\right) \leq  \frac{1}{2^{m+2}}v(x_{j_0}) \  \  \mbox{ for all } \  \  (k_j) \in K_m)
.\eqno(9)
$$
Even if $\varphi$ is not a homomorphism, one has
$$
\varphi(\xi+ \eta)= \varphi(\xi)+\varphi(\eta), \;\; \mbox{ whenever }\;  \; \xi=(k_j), \eta=(l_j) \in K_m,\eqno(10)
$$
where $\xi+ \eta= (k_j + l_j)\in K_{m-1}$.

Fix $m\in \Z$. In order to show that $\varphi\restriction_{K_m}$ is continous, fix $(k_j)\in K_m$ and $\eps>0$. There exists $n\in\N$ such that $\frac{1}{2^{m+1}} v(x_n)<\eps$.
For $(l_j)\in (k_j) + W_n$ we have $\varphi((l_j))\in
%\varphi((k_j)) + B_{\frac{v(x_{j_0})}{2^{m+2}}}(0)+ B_{\frac{v(x_{j_0})}{2^{m+2}}}(0)   \subseteq
\varphi((k_j)) + B_{\eps}(0)$
%v\left (\varphi((k_j)-\varphi((l_j))\right)\leq v\left(\sum_{j>j_0}k_j x_j\right)+ v\left(\sum_{j>j_0}l_jx_j\right)\leq  \frac{1}{2^{m+2}} v(x_{j_0})+ \frac{1}{2^{m+2}}v(x_{j_0}) =\frac{1}{2^{m+1}}v(x_{j_0})<\eps
%$$
by (9), which shows that $\varphi$ is continuous.

 In order to show that $\varphi$ is injective, we show the following stronger statement, that will  be necessary bellow:

 \begin{claim}\label{Claim3} If $(k_j)\in K_{-1}$ and $(l_j)\in K_0$ with  $(k_j)\not=(l_j)$, then
$\varphi((k_j))\ne \varphi((l_j))$.
\end{claim}

Assume for a contradiction that $\varphi((k_j))=\varphi((l_j))$ with $(k_j)\not=(l_j)$.
Fix $m$ minimal with $k_{m}\not=l_{m}$. Then
 $$
 (l_{m}-k_{m})x_{m}=\sum_{j>m}k_j x_j - \sum_{j>m}l_j x_j
 $$
 with $|k_{j}-l _{j} |  \leq 3\cdot 2^{j}<2^{j+2}$ for  all $j\geq m$ (as $k_{j} \leq 2^{j+1}$ and $l_{j} \leq 2^{j}$). Hence the definition of $\eps_{m}$ gives
$$
\eps_{m}\leq v((l_{m}-k_{m})x_{m})\le  |k_{m+1}-l _{m+1} |v(x_{m+1})+  v\left(\sum_{j>m+1}k_j x_j\right) + v\left(\sum_{j>m+1}l_j x_j\right).\eqno(11)
$$
To the second and the third term in the right hand side of (11) we may apply (9) with $m=0$, respectively $m=-1$ and
 $n=m+1$ and obtain
%$(l_j)\in K_{0} \cap W_{m+1}$. So from
%Claim \ref{Claim2}
% (more precisely, (11))
%we get
$$
v\left(\sum_{j>m+1}k_j x_j\right) + v\left(\sum_{j>m+1}l_j x_j\right)\leq
\frac{1}{2}v(x_{m+1}) + \frac{1}{4}v(x_{m+1})= \frac{3}{4}v(x_{m+1}).\eqno(12)
$$
Since $|k_{m+1}-l _{m+1} |  \leq 3\cdot 2^{m+1}$ yields
$
|k_{m+1}-l _{m+1} | v(x_{m+1})\leq 3\cdot 2^{m+1}v(x_{m+1}),  %\eqno(14)
$
by (11) and (12), we get $\eps_{m}\le 3\cdot 2^{m+1}v(x_{m+1}) + \frac{3}{4}v(x_{m+1})$.
Along with (8), applied with $n=m$, we get
$\eps_{m}\le 3\cdot 2^{m+1}v(x_{m+1}) + \frac{3}{4}v(x_{m+1})<2^{m+3}v(x_{m+1})<\eps_m$, a contradiction.
%$\eps_{m}\le \frac{3}{4}(1 +  \frac{1}{2^{m+3}} ) \eps_{m} < \eps_{m}$, a contradiction.
This proves Claim \ref{Claim3}.

\smallskip

From Claim \ref{Claim3} we conclude that $\varphi\restriction_{K_0}$ is a continuous bijective mapping. Since $K_0$ is compact, $\varphi\restriction_{K_0}: K_0 \to \wt X= \varphi(K_0)$ is a homeomorphism which implies in particular that $\wt X$ is compact and totally disconnected. Since $\varphi$ is an extension of $\varphi_{\mathbf x}$
and since $\fZ_0 \cap K_0$ is dense in $K_0$, we deduce that  $ X = \varphi_{\mathbf x}(\fZ_0  \cap K_0)=\varphi(\fZ_0  \cap  K_0)$ is dense in $\wt X= \varphi(K_0)$. Since the latter set is compact, it must be closed in $G$. Therefore,
$\wt X$ coincides with the closure of $X$.

Next we claim that
%$$
%(1/2)\wt X =\varphi(K_1).\eqno(16)
%$$
% Indeed,  fix $x=\sum_{j\ge 0}k_j x_j \in (1/2)\wt X$, $(k_j) \in K_0$. By   definition $2x=\sum_{j\ge 0}2k_j x_j\in X$. This means there exists a representation $2x=\sum_{j\ge 0}l_jx_j$ with  $(l_j) \in K$. Since obviously $(2k_j)\in K_{-1}$,  from $\varphi((l_j)))= \varphi((2k_j))$ we deduce that  $2k_j=l_j$ for every $j$. Moreover, $(l_j) \in K$ implies
% $(k_j) \in K_1$, so $x\in \varphi(K_1)$. Therefore, $(1/2)\wt X  \subseteq \varphi(K_1)$. Since the other inclusion is obvious, this proves (16).

%
%%(since $ \varphi(\xi+ \eta)= \varphi(\xi)+\varphi(\eta)$ for $\xi, \eta \in K_1$, as mentioned above).
%%$K\ni \varphi(\xi+ \eta)= \varphi(\xi)+\varphi(\eta)= x + y$. So
%Therefore, $\gamma(\wt X)\le 2$. Since $x_0\in X $ and $2 x_0\notin X$, we obtain
%$\gamma(\wt X)=2$.
%To this end it suffices to prove that
$$
(1/2^m) \wt X=\varphi(K_m)= \left\{\sum_{j=0}^\infty k_jx_j:\  k_j\in\Z,\
\sum_{j=0}^\infty\frac{|k_j|}{2^j}\le \frac{1 }{2^m}\right\},\eqno(13)
$$
as in the case of the set $X$ in Lemma \ref{LemmaUFSS}.

 The inclusion $\supseteq$ in (13) is obvious. We prove the following stronger version of the reverse inclusion by induction:
$$
\mbox{if }\; x=\sum_{j=0}^\infty k_jx_j, \mbox{ with} \;
(k_j) \in K_0,
%\sum_{j=0}^\infty \frac{|k_j|}{2^j}\le 1
\  \mbox{ then } \; \; x\in (1/2^m)\wt{X} \;\; \Longrightarrow \;\; (k_j) \in K_m
%\sum_{j=0}^\infty \frac{|k_j|}{2^j}\le \frac{1 }{2^m}
.\eqno(14)
$$
For $m=0$ the assertion is trivial. So suppose (14) holds true for $m$ and let $x=\varphi((k_j))$, with $(k_j)\in K_0$ belong to $(1/2^{m+1})\wt X$.
%\sum_{j=0}^\infty k_jx_j \in (1/2^{m+1})X$,  with $\sum_{j=0}^\infty\frac{|k_j|}{2^j}\le 1$.
Since $x,2x\in (1/2^m)\wt X$, by the induction hypothesis,
$(k_j)\in K_m$. Moreover, there exists a  representation $2x=\varphi((l_j))$ with $(l_j)\in K_m$.  Then $\varphi((2k_j))=\varphi((l_j))$. Since, $(2k_j) \in K_{-1}$ and $(l_j)\in K_0$, from Claim \ref{Claim3}
we conclude that  $2k_j=l_j$ for all $j$ and hence $\sum_{j=0}^\infty\frac{|k_j|}{2^j}\le \frac{1}{2^{m+1}}.$ This proves (14), and consequently also (13).

In particular, from (13) we get $(1/2)\wt X =\varphi(K_1)$. Since $K_1 + K_1 \subseteq K_0$ from Remark \ref{remark:NIndep}(b),
this combined with with (13), gives $(1/2)\wt X+(1/2)\wt X\subseteq \wt X$. Hence $\gamma(\wt{X})\leq 2$.

It remains to note that (14) implies also $\wt X_\infty=\{0\}$. Hence ${\cal T}_{\wt{X}}$ is a UFSS topology coarser than ${\cal T}_X$ (as $X\subseteq \wt{X}$), so
it is non-discrete. In particular, $\wt  X \ne \{0\} =\wt X_\infty$, so $\gamma(\wt{X})= 2$.
From (9) and (13) we conclude that ${\cal T}_{\wt{X}}$ is finer than the original topology of $G$.
\end{proof}

In order to characterize those  abelian metrizable groups which admit a (strictly) finer UFSS group topology, we need the following definition which will characterize these groups.

\begin{definition}
An abelian  topological group $G$ is called {\em locally bounded} if there exists some $n\in\N$ such that
the subgroup $G[n]=\{x\in G:\ nx=0\}$ is open.
\end{definition}

\begin{remark}\label{remlocbounded}
$G$ is locally bounded iff it has a neighbourhood $U$ in which all elements are of bounded order. Obviously, a   metric abelian  topological group $G$ is not locally bounded iff
there exists a null sequence $x_n \to 0$ such that $o(x_n)\to \infty$.

A locally compact abelian group $G$ is locally bounded iff it has an open compact subgroup of
finite exponent. Ideed, assume that $G$ is a locally compact, locally bounded abelian group. For suitable
$n \in\N$ the
subgroup $G[n]$ is open. By the structure theorem for locally compact abelian groups, $G[n]$
contains an  open subgroup $K$. It is clear
that $K$ is open in $G$ and of finite exponent. The converse implication is trivial.

\end{remark}

\begin{theorem}\label{Lydia_Theorem} Let $(G,d)$ be an abelian,  metrizable, non--discrete group. The following assertions are equivalent:
\begin{enumerate}
\item[(i)]  $G$ is not locally bounded;
\item[(ii)]  there exists a   finer non--discrete UFSS group topology ${\cal T}_X$  on $G$;
\item[(iii)] there exists a  strictly finer non--discrete UFSS group topology ${\cal T}_Y$  on $G$;
\item[(iv)] there exists an almost independent sequence in $G$.
  \end{enumerate}
%In case the metric $d$ is complete, the sets $X$, respectively $Y$ can be modified  to be compact in the metric topology.
\end{theorem}

\begin{proof} (iii) $\Longrightarrow$ (ii) is trivial.

(ii) $\Longrightarrow$ (i): Let $G$ be  locally bounded, this means there exists $n\ge 1$ such that the subgroup $G[n]$ is open. Suppose there
exists a UFSS topology ${\cal T}_X$  on $G$ with distinguished neighborhood $X$ which is   finer than the
 topology $\tau$ induced by the metric. We may assume that
$X\subseteq G[n]$, because otherwise we can replace $X$ by $X\cap G[n]$.  Then
$(1/n)X=X_\infty$. Since we assumed ${\cal T}_X$ to be finer than the original topology and hence Hausdorff,
$\{0\}=X_\infty =(1/n)X$. This implies that ${\cal T}_X$ is discrete. So the only finer UFSS group topology on $G$ is the discrete one.

(iv) $\Longrightarrow$ (iii): this is covered by the preceding lemma.

(i) $\Longrightarrow$ (iv): Assume now that $G$ is not locally bounded. We have to show that there exists
an almost independent sequence $(x_n)$
 of elements in $G$.
   This will be done by induction. For $n=0$ condition (4) is equivalent to $3x_0\not=0\not=4x_0 $.
So fix an element $x_0\in G$ of order greater than $ 4$. Assume that $x_0,\ldots,x_n$ have already been chosen to satisfy (4). Define $
\eps_n$ by (7). Then choose $x_{n+1}$ with $v(x_{n+1})< \eps_n/2^{n+3} $  and
  $o(x_{n+1})>2^{n+3}$. To check that this works, let $x=\sum_{j=0}^{n+1}a_jx_j$ with $(a_j)\not=(0)$ and $|a_j|\le 2^{j+2}$ for $j=0,1,\ldots, n+1$. If $a_0=\ldots=a_n=0$ then $x\not=0$, since $|a_{n+1}|\le 2^{n+3}<{\rm o}(x_{n+1}).$ Otherwise, we have
$
v(x)\ge v\left(\sum_{j=0}^na_jx_j\right)-v(a_{n+1}x_{n+1})\ge \eps_n - 2^{n+3}v(x_{n+1})>0 ,
$
so $x\not=0$. By the choice of each $x_n$,  the sequence is almost independent.
\end{proof}

\begin{remark}
Let us note that for the set $Y$ constructed in the proof, the strictly finer non-discrete UFSS topology ${\cal T}_Y$ is still locally unbounded. So to the group
$(G, {\cal T}_Y)$ the same construction can be applied to provide an infinite  strictly increasing chain of non-discrete UFSS topologies $
{\cal T}_Y= {\cal T}_{Y_0} < {\cal T}_{Y_1} < \ldots < {\cal T}_{Y_n} < \ldots.$
Hence in the theorem one can also add a stronger property (v) claiming the existence of such a chain.
\end{remark}

\begin{corollary}\label{Navarra} Let $(G,d)$ be a complete abelian,  metrizable non locally bounded group.
Then there exists a compact totally disconnected GTG set $X$ of $G$, such that ${\cal T}_X$ is
 a    finer non-discrete UFSS group topology   on $G$.
\end{corollary}

\begin{proof} According to the above theorem  $G$ admits an almost independent sequence $(x_n)$. \end{proof}

E. Hewitt \cite{H} observed that the group $\T$ and the group $\R$ have the property that the only stronger
locally compact group topologies are the discrete topologies. Since locally minimal topologies generalize the
locally compact group topologies, this suggests the following question:  {\em Do the groups $\T$ and $\R$
admit stronger non-discrete locally minimal topologies?} The next corollary answers this question in a strongly
positive way. Namely,
the class of all non-totally disconnected locally compact metrizable abelian groups
(in place of $\T$ and  $\R$ only) and for the smaller class of UFSS topologies (in place of locally minimal
topologies).

\begin{corollary}\label{Hewitt}
A locally compact abelian metrizable group $G$ has  a strict UFSS refinement iff $G$ contains
no open compact subgroup of finite exponent.

This happens for example, if $G$ is not totally disconnected.

\end{corollary}

\begin{proof} The first assertion is obvious when $G$ is discrete, so we assume that $G$ is non-discrete in the sequel.

  According to \ref{Lydia_Theorem}, $G$ has a strict, non-discrete UFSS refinement iff $G$ is
locally bounded, which, by \ref{remlocbounded}, is equivalent to the existence of a compact open subgroup
of finite exponent.

In order to prove the second statement it is sufficient to show that every group $H$ which has an open compact
subroup $K$ of finite exponent is totally disconnected. The  connected component $C$   of $H$   is
contained in $K$ and hence bounded. On the other hand side, as every compact abelian connected group, $C$
is divisible. This
implies that $C$ is trivial and hence $H$ is totally disconnected.

%Assume that $G$ contains an open subgroup $K$ of finite exponent. Then $G$ is obviously locally bounded,
%so  $G$ has no strict UFSS refinement according to the previous theorem.
%Conversely, if $G$ has no strict UFSS refinement, then the previous theorem implies that $G$ is locally bounded.
%Let $G[n]$ be a bounded open subgroup of $G$. Then $G[n]$ is totally disconnected locally compact;
%indeed, the connected compontent of the bounded locally compact group $G[n]$ is divisible (as the connected component
%of every
%
% hence $G[n]$ contains an open
%compact subgroup $K$ of finite exponent.\NL\footnote{I suggest the following reformulation of the last sentence:
%Applying the structure theorem for locally compact groups, we obtain that $G$ (and hence also $G[n]$ has an open compact
%subgroup which is of exponent $\le n$.}
\end{proof}

%\begin{corollary}\label{Hewitt}
%Every non-totally disconnected locally compact abelian metrizable group has  a strict UFSS refinement.
%\end{corollary}
%
%\begin{proof} Let $G$ be a non-totally disconnected locally compact abelian metrizable group.
%Clearly, $G$ does not contain any open compact subgroup $N$ of finite exponent, since such a group
%$N$ would be zero-dimensional. Indeed, assume $nN= 0$ for some $n \in \N$.
%Since the connected component $N_0$ of $N$ must be divisible
%(as every compact connected group is divisible), we conclude that  $N_0 = nN_0 = 0$ is trivial.
%  Since totally disconnected compact groups are  zero-dimensional, we conclude that $N$ is  zero-dimensional.
%As $N$ is an open subgroup of $G$, this yields that $G$ is zero-dimensional too, a contradiction.
%Now the above corollary applies.
%\end{proof}

Now comes the topology-free version of Theorem \ref{Lydia_Theorem}:

\begin{corollary}\label{LastCorollary} For an abelian group $G$ TFAE:
\begin{enumerate}
\item[(i)]  $G$ is not bounded;
\item[(ii)]  $G$ admits a non--discrete UFSS group topology;
\item[(iii)] there exists a nearly independent sequence in $G$.
\end{enumerate}
\end{corollary}

\begin{proof}
The implication (iii) $\Longrightarrow$ (ii) was proved in Lemma \ref{LemmaUFSS}.

To prove the implication (ii) $\Longrightarrow$ (i) assume $G$ admits a non--discrete UFSS group topology
${\cal T}$ with distinguished neighborhood $U$ of 0. Then
for every $n\in \N$ the set $(1/n)U$ is a ${\cal T}$-neighborhood of 0, hence $(1/n)U\ne \{0\}$.
If $nG$ were $\{0\}$
for some $n\in \N$, then $(1/n)U=U_\infty=\{0\}$ which is a contradiction. So $G$ is unbounded.

%Thus $nU \ne 0$, and consequently, $nG \ne 0$ as well.
%This proves that $G$ is not bounded.

To prove the implication (i) $\Longrightarrow$ (iii) pick a countable subgroup $H$ of $G$ that is still not bounded. Since $H$ is countable,
there exists an injective homomorphism $j: H \to \T^\N$. Denote by $d$ the metric induced on $H$ by this embedding. Then $(H,d)$ is a metric precompact group, hence
it is not discrete. Moreover, for no $n\in \N$ the subgroup $H[n]$ is open. Indeed, if $H[n]$ were open, then by the precompactness of $H$ it has finite index in $H$. Hence
$mH \subseteq H[n]$ for some $m\in \N$. Therefore, $mn H = 0$, a contradiction. This argument proves that no subgroup $H[n]$ ($n\in \N$) is open in $H$. Hence, $(H,d)$ is
not locally bounded. Then $H$ contains an almost independent sequence $(x_n)$ by the above theorem. Clearly, this is also a  nearly independent sequence in $H$, and consequently, also in $G$.
\end{proof}

%%%%%%%%%%%%%%%%%%%%%%%%%%%%%%%

\section{Locally GTG groups}\label{newGTG}

\subsection{Locally GTG groups and their properties}

\begin{definition}[V. Tarieladze, oral communication]  \label{defGTG}  We say that a Hausdorff
topological abelian group $G$ is {\em locally GTG} if it admits a basis of
neighborhoods of the identity formed by GTG subsets of $G$.
\end{definition}

\begin{example}\label{Exampl0TG}
 \begin{itemize}
 \item[(a)] Every UFSS group is locally GTG. In particular $\R$ and
 $\T$ are locally GTG.
 \item[(b)] Every locally convex space is locally GTG.
\item[(c)] Assume that $G$ is a bounded abelian group with exponent $m$. If $U$
is a GTG neighborhood of $0$ in some group topology $\tau$ of $G$, then
$U_\infty=(1/m)U$ is a $\tau$-neighborhood  of $0$. Therefore,
\begin{itemize}
  \item[(c$_1$)] $(G,\tau)$ is locally GTG precisely when $(G,\tau)$ is
  linearly topologized.
  \item[(c$_2$)] $(G,\tau)$ is UFSS precisely when $(G,\tau)$ is discrete.
\end{itemize}
\end{itemize}
\end{example}

\begin{example}\label{pseudoconvex}
 A topological vector space is said to be
{\em locally pseudoconvex} if it has a basis of pseudoconvex neighborhoods of zero. A topological vector space is locally GTG as a
topological abelian group if and only if it is locally pseudoconvex.
\end{example}

\begin{proof}
Applying \ref{pcGTG}, it suffices to show that a topological
vector space which is locally GTG has a neighborhood basis
consisting of balanced GTG sets. So fix a GTG neighborhood $A$ and
define $B:=\{a\in A:\ [0,1]a\subseteq A\}$. It is straightforward
to prove that $[0,1]B\subseteq B$ and it is a well known fact that
$B$ is a neighborhood of zero. Let us prove that $B$ is GTG. Since $A$ was assumed to be GTG, there exists $n\in \N$ such that
$(1/n)A+(1/n)A\subseteq A$. Observe that $(1/n)B=\frac{1}{n} B$. We shall
show that $\frac{1}{n} B+\frac{1}{n}B\subseteq B$. So fix $a,b\in\frac{1}{n} B$ and $t\in
[0,1]$. Let us see that $t(a+b)$ belongs to $A$: $t(a+b)=ta+tb \in \frac{1}{n} B+\frac{1}{n}B\subseteq (1/n)A+(1/n)A\subseteq A$.
\end{proof}

\begin{example} \label{elecero}
Local GTGness may seem to be a too mild
property, but there exist natural examples of abelian topological groups
lacking
it. Consider the topological vector space $G=L^{0}$ of all classes of
 Lebesgue measurable functions $f$ on $[0,1]$ (modulo almost everywhere equality) with the topology of convergence in
measure. This topology can be defined by the invariant metric
$$
d(f,g)=\int_0^1\min\{1,|f(t)-g(t)| \}dt$$ (for details see for instance
\cite[Ch. 2]{Kal}). It is known that $L^0$ is not locally pseudoconvex and hence, by Example \ref{pseudoconvex}, it is not locally GTG as a topological group.
\end{example}
Here we collect several properties of locally GTG groups.

\begin{proposition}\label{permGTG}
\begin{itemize}
\item[(a)] Every subgroup of a locally GTG group is locally GTG.
\item[(b)] A group with an open locally GTG subgroup is locally GTG.
\item[(c)] The product of locally GTG groups is locally GTG.
\item[(d)] Quotient groups of locally GTG groups need not be locally GTG.
\item[(e)] Every group locally isomorphic to a locally GTG group is locally GTG. In particular, if  a topological group $G$ admits a  non-trivial locally GTG open subgroup, then $G$ is locally GTG.
\end{itemize}
%%%%%%%%%%%%%%%%%%%%%%%%%%%%

\end{proposition}

\begin{proof}
(a) is a consequence of Lemma \ref{inv_im_GTG}(a) and Example
\ref{ejemplosgtg}(b). (b) follows from the fact that any basis of neighborhoods of zero in the open subgroup is a basis of neighborhoods of zero in the whole group. (c) is a consequence of \ref{inv_im_GTG}(c).
(d) Let $G$ be a Hausdorff group which is not locally GTG. $G$ is a
 quotient of the free abelian topological group $A(G)$ (\cite{Markov}).
  The free locally convex space $L(G)$ is locally GTG according to
  \ref{Exampl0TG}(b). According to  a result of Uspenskij and Tkachenko
   (\cite{Us} and \cite{Tka}) the free abelian topological group $A(G)$
    is a subgroup of $L(G)$ and hence, due to (a), also locally GTG.
     This proves (d). (e) is straightforward using  Lemma \ref{inv_im_GTG}(a).
\end{proof}

  Now we obtain another large class of examples:

\begin{example}\label{prec>Loc-GTG}
\begin{itemize}
\item[(a)]  Every precompact abelian group is locally GTG.
Indeed, every precompact abelian group is (isomorphic to) a subgroup
of a  power of $\T$, so items (a) and (c) of Proposition
\ref{permGTG} and item (a) of Example \ref{Exampl0TG} apply.
\item[(b)]  Every locally compact abelian group is locally GTG. Indeed,
every locally compact abelian group has the form $G = \R^n \times
G_0$, where $n\in \N$ and $G_0$ contains an open compact subgroup.
Then $G_0$ is locally GTG by item (a) and item (e) of  Proposition \ref{permGTG}, while $\R^n$ is UFSS, so
locally GTG. Now item (c) of Proposition \ref{permGTG} applies.
\end{itemize}
\end{example}

The connection between locally GTG and UFSS groups is the following:

\begin{theorem}\label{Prop2GTG}\begin{itemize}
\item[(a)] If $U$ is a GTG subset of an abelian group $G$, the quotient group
$G_U:=(G,{\mathcal T_U})/ U_{\infty}$  is UFSS when equipped with the quotient
topology of ${\mathcal T}_U$.
\item[(b)] Every locally GTG group $G$ can be embedded into a product of UFSS
groups.
\item[(c)] A group topology $\tau$ on an abelian group $G$ is a supremum of
UFSS topologies on $G$ iff $\tau$ is NSS and locally GTG.
\item[(d)] If a group topology $\tau$ on an abelian group $G$ is a supremum
of  a family ${\mathfrak T}=\{\tau_i:i\in I\}$ of UFSS topologies on $G$,
then $\tau$ is UFSS iff $\tau$ coincides with the supremum of a finite
subfamily of ${\mathfrak T}$.
\end{itemize}
\end{theorem}

\begin{proof} (a)  Let $U$ be a GTG subset of $G$. Since $U_{\infty}$ is the
${\mathcal T}_U$-closure of $\{0\}$, we can consider  the Hausdorff
quotient group $G_U=(G,{\mathcal T_U})/ U_{\infty},$ and  the canonical
epimorphism $\varphi_U:G\to G_U$.
\par
 Let $m\in \N$ be such that $(1/m)U +(1/m)U \subseteq U.$  Let us show
 that for every $n\in \N$
$$
(1/n)\varphi_U((1/m)U ) \subseteq \varphi_U((1/n)U ).
$$
Indeed, fix an element $\varphi_U(x)\in (1/n)\varphi_U((1/m)U ) .$ Then
for every $k\in \{1,\dots, n\}$, $kx\in
(1/m)U+U_{\infty}\subseteq U,$ hence $\varphi_U(x)\in \varphi_U((1/n)U).$

 This shows that $G_U$ is a UFSS group with distinguished neighborhood
 $\varphi_U((1/m)U)$.

(b)  Let ${\mathcal U}$ be a basis of neighborhoods of zero in $G$ formed
by GTG sets. The homomorphism
$$
\Phi: G\to \prod_{U\in {\mathcal U}}G_U,\quad \Phi(x)=(\varphi_U(x))_{U\in
{\mathcal U}}
$$
 is injective and continuous. Fix $U\in {\mathcal U,}$ and let $m\in \N$
 be such that $(1/m)U+(1/m)U\subseteq U.$ Then $(1/m)U+U_{\infty }\subseteq U,$
 from which we deduce
$$
\Phi(U) \supset \Phi(G) \cap \Big(\big(\prod_{U'\in {\mathcal U}
\setminus \{ U\}} G_{U'}\big)\times \varphi_U((1/m)U  ) \Big).
$$
This implies that $\Phi$ is open onto its image.

(c) It is clear that every supremum of UFSS topologies is both NSS and
locally GTG. Conversely, if $G$ is locally GTG, its topology is the
supremum of the family of topologies $\{{\mathcal T}_U\}_{U\in {\mathcal U}}$
where ${\mathcal U}$ is a basis of neighborhoods of zero. If moreover $G$ is
 NSS, we may assume that no neighborhood in ${\mathcal U}$ contains nontrivial
  subgroups, and in particular the topologies ${\mathcal T}_U$ are UFSS.

(d)
The sufficiency is obvious from Remark \ref{RemUFSS0} (a). To prove the
necessity let us assume that $\tau=\sup \tau_i$ is UFSS. Then
there exists a distinguished $\tau$-neighborhood  $W$ of 0 such that
$\tau={\mathcal T}_W$. There exists
a finite subset $J \subseteq I$ and $\tau_j$-neighborhoods $U_j$ of 0
for each $j\in J$ such that
$\bigcap_{j\in J}U_j\subseteq W$. We can assume without loss of generality
that $U_j$ is a distinguished
neighborhood  of 0 in $\tau_j$ for each $j\in J$. Then $(1/n)W\in  \sup_
{j\in J}\tau_j$ for every $n$.
Hence $\tau={\mathcal T}_W \leq \sup_{j\in J}\tau_j$. The inequality $\tau=
\sup_{j\in I}\geq \sup_{j\in J}\tau_j$ is trivial. This proves that $\tau=
\sup_{j\in I}\tau_j$.
\end{proof}

\begin{remark}\label{nssneeded}

Note that ``NSS"{} is needed in (c) above; any nonmetrizable compact
abelian group is locally GTG (see Example \ref{prec>Loc-GTG}(a))
%Proposition \ref{Bemerkung} below)
 but its topology is not a supremum of UFSS topologies.
\end{remark}

\begin{corollary} \label{complGTG} The class of locally GTG abelian
groups is stable under taking completions. \end{corollary}

\begin{proof} By Theorem \ref{Prop2GTG}(b), every locally GTG group
$G$ can be embedded into a product $\prod_ i G_i$ of UFSS groups
$G_i$. By Proposition \ref{perm_prop_3} (a), the completion
$\wt{G}_i$ of the UFSS group $G_i$ is UFSS.  So the completion
$\wt{G}$ of $G$ embeds into the product $P=\prod_i \wt{G}_i$ of UFSS
groups. By Proposition \ref{permGTG} (c) $P$ is locally GTG, so $\wt{G}$
is locally GTG by Proposition \ref{permGTG} (a).
\end{proof}

\begin{theorem}\label{MJX} A Hausdorff abelian topological group
$(G,\tau)$ is a UFSS group if and only if $(G, \tau)$ is locally
minimal, locally GTG and NSS.
\end{theorem}

\begin{proof} Suppose that $(G, \tau)$ is a  UFSS group with distinguished
neighborhood  $U$.  Then $(G,\tau)$ is $U$--locally minimal
according to Facts \ref{UFFS_locmin}(a), locally GTG according to Example
\ref{Exampl0TG}(a) and $U$ does not contain any nontrivial subgroup.

Conversely, let $(G, \tau)$ be locally minimal, locally GTG and NSS.
There exists a neighborhood of zero $U$ which is a GTG set,
witnesses local minimality and does not contain nontrivial subgroups.
The group topology ${\mathcal T}_U$ generated by $U$ is Hausdorff and
coarser than $ \tau$; since $U$ is one of its zero neighborhoods, it
coincides with  $ \tau$.
\end{proof}

%%%%%%%%%%%%%%%%%%%%%%%%%%%%%%%%%%

\subsection{Locally minimal, locally GTG groups}\label{locmin_locgtg}

In this section we will give various properties of   locally minimal
locally GTG groups. Most of our results  are based on the following proposition which allows us to find large, in appropriate sense, minimal subgroups in a locally minimal group.

\begin{proposition}\label{minsubgr}({\rm \cite{DM}}) Let $G$ be a
$U$--locally minimal group  and let $H$ be a closed central subgroup
of $G$ such
that $H+V\subseteq U$ for some neighborhood $V$ of $0$ in $G$. Then $H$
is   minimal.
\end{proposition}

\begin{theorem}\label{GTG+locMin}

If $G$ is a  $U$--locally minimal abelian group where $U$ is a GTG set,
then $U_\infty$ is a minimal subgroup.
\end{theorem}

\begin{proof} Proposition \ref{definicion_bis_gtg} implies $U_\infty+(1/m)U
\subseteq U$ for some $m\in \N$. Then, Proposition  \ref{minsubgr}
 immediately
gives us that $U_\infty$ is a minimal subgroup.
\end{proof}

%This corollary justifies our choice to consider mostly $U$--locally
%minimal abelian groups  where $U$ is a GTG neighborhood of $0$.

One may ask whether GTG is needed in the above corollary (see Question
\ref{LocMin_vs_locGTG}). The problem is that without this assumption,
the intersection $U_\infty$ need not be a subgroup (although it is always a
union of cyclic subgroups), as it happens in Example \ref{elecero}.

 It easily follows from Theorem \ref{GTG+locMin} that every locally minimal
locally GTG abelian group contains a minimal, hence precompact, $G_\delta$-subgroup
(note that the subgroup $U_\infty$ in Theorem \ref{GTG+locMin} is a $G_\delta$-set).
Now we provide a different proof of this fact, that makes no recourse to local GTG-ness.

\begin{proposition}\label{LAST_Proposition} Every locally minimal abelian group contains a minimal, hence precompact, $G_\delta$-subgroup.
\end{proposition}

\begin{proof} Let $U$ witness local minimality of the group $G$. As in the proof of Proposition \ref{loc_min+NSS>metr},  it is possible to construct inductively a sequence $(V_n)$ of symmetric
neighborhoods of $0$ in $\tau$ which satisfy $V_n+V_n\subseteq V_{n-1}$
(where $V_0:=U\cap -U$). It is easy to see that $H= \bigcap_{n\in\N}V_n$ is a subgroup of $G$, contained in each $V_n$.
In particular, $H + V_1 \subseteq V_0 \subseteq  U$.  Now Proposition  \ref{minsubgr} immediately gives us that $H$ is a minimal subgroup.
\end{proof}

Let us note that the minimal $G_\delta$-subgroup obtained in this proof is certainly contained in  the subgroup $U_\infty$, provided
$U$ is a GTG set (as $H\subseteq U$ and $U_\infty$ is the largest  subgroup contained in $U$). However, this argument has the advantage
to require weaker hypotheses.

The next corollary shows that non-metrizable complete locally minimal abelian groups contain large compact subgroups.

\begin{corollary}\label{completeLocMinLocGTG} Every complete locally minimal abelian group contains a compact $G_\delta$-subgroup.
\end{corollary}

\begin{proof} Follows directly from Proposition \ref{LAST_Proposition}.
%subgroup $U_\infty$ is a $G_\delta$-set.
\end{proof}

\begin{corollary}\label{lin_top} Let $(G,\tau)$ be either
\begin{itemize}
   \item[(a)] a linearly topologized abelian group, or
   \item[(b)] a bounded locally GTG abelian group.
\end{itemize} Then $G$ is locally minimal iff $G$ has an open minimal subgroup.
\end{corollary}

\begin{proof} If $G$ has an open minimal subgroup, then $G$ is locally minimal
(Proposition \ref{locally minimal open subgroup}). Conversely, suppose that $G$ is
locally minimal.

(a) Let $V$ be an open subgroup of $G$ witnessing local minimality of $G$.
Then $V+V\subseteq V$, so $V$ is minimal by Proposition \ref{minsubgr}.

(b)  Let $G$ be $U$--locally minimal for a GTG neighborhood $U$.
According to Theorem \ref{GTG+locMin}, $U_\infty$ is a minimal
subgroup of $G$. For the exponent $m$   of $G$, we obtain
$(1/m)U=U_\infty$ and hence $U_\infty$ is open.
\end{proof}

If the algebraic structure of a group is sufficiently well
understood, Theorem \ref{GTG+locMin} helps to characterize locally
minimal group topologies. As an example we describe the locally
minimal locally GTG   topologies on $\Z$. Let us recall that the
minimal topologies on $\Z$ are precisely the $p$-adic ones (Prodanov
\cite{P1}).

\begin{example}\label{Ex_LA1} Let $(\Z,\tau)$ be  a locally minimal locally
GTG  group topology. Then either

\begin{itemize}
\item[(a)] it is UFSS; or
\item[(b)] $(\Z,\tau)$ has an open minimal subgroup; more precisely,
there exists a prime number $p$ and $n\in\Z$ such that $(np^m\Z)_{m\in\N}$
forms a neighborhood basis of the neutral element.
\end{itemize}
 Indeed, if $\tau$ is not UFSS Theorem \ref{MJX} gives that it is not NSS, and then,
 Example \ref{Zlinear}(d)
 says that $\tau$ is a non-discrete linear topology.
 We
 apply now Corollary \ref{lin_top} and we obtain that $G$ contains an open
 minimal subgroup $N$. Let $N=n\Z$ for some $n\not=0$. Then  the minimality of
  $N$ implies that for a suitable prime $p$, a neighborhood basis of $0$ in
 $n\Z$ is given by the sequence of subgroups $(np^m\Z)_{m\in\N}$ ((2.5.6) in
  \cite{DPS}).
\end{example}

\begin{proposition}\label{Do1loc}
Products of locally minimal (abelian precompact) groups are in
general not locally minimal, namely the group of integers with the
$2$-adic topology  $(\Z,\tau_2)$ is minimal and hence locally
minimal, but the product $(\Z,\tau_2)\times (\Z,\tau_2)$ is not
locally minimal.
\end{proposition}

\begin{proof}
Suppose that $(\Z,\tau_2)\times (\Z,\tau_2)$ is $U$--locally
minimal. We may assume that $U=2^n\Z\times 2^n\Z$. By
\ref{minsubgr}, the closed subgroup $U$ is minimal. But $U$ is
topologically isomorphic to $(\Z,\tau_2)\times (\Z,\tau_2),$ which
yields a contradiction.
\end{proof}

 According to Corollary \ref{lin_top}(b) the bounded locally minimal
locally GTG abelian groups have an open minimal subgroup. Now we use
this fact to describe the bounded abelian groups that support a
non-discrete locally
 minimal and locally GTG group topology:

\begin{theorem}\label{NewThm1}
 Let $G$ be a bounded abelian group. Then the
 following assertions are equivalent:
\begin{itemize}
  \item[(a)] $|G|\geq \cont$;
  \item[(b)] $G$ admits a non-discrete locally  minimal and locally
            GTG group topology;
  \item[(c)] $G$ admits a non-discrete locally  compact metrizable group
            topology.
\end{itemize}
\end{theorem}

\begin{proof} To prove the implication (a) $\Rightarrow$ (c)  use Pr\" ufer's
theorem to deduce that $G$ is a direct sum of cyclic subgroups.
Since $G$ is bounded, there exists an $m>1$ such that $G$ has as a direct
summand a subgroup $H \cong \bigoplus_\cont \Z(m)\cong \Z(m)^\omega$.
Since  $\Z(m)^\omega$ carries a metrizable compact group topology, one can
build a non-discrete locally  compact metrizable group topology on $G$ by
putting on $H$ the topology transported by the isomorphism $H\cong \Z(m)^\omega$
and letting $H$ to be an open subgroup of $G$.

(c) $\Rightarrow$ (b) Let $\tau$ be a non-discrete locally  compact
group topology on the group $G$. According to \ref{Exa_LocMin}(b)
and \ref{prec>Loc-GTG}, $\tau$ is locally minimal and locally GTG.

(b) $\Rightarrow$ (a) Assume that $|G|<\cont$. By Corollary
\ref{lin_top} there exists an open minimal subgroup $H$ of $G$. As
$|H|<\cont$, we conclude that $r_p(H)<\infty$ for all primes $p$
(see \cite[Cor. 5.1.5]{DPS}). Since $H$ is a bounded abelian group
we conclude that  $H$ is finite. Since $H$ is open in $G$, $G$ is
discrete, a contradiction.
\end{proof}

\section{Open questions}
\begin{question}\label{QclGTG}
Is the closure of every GTG set in a topological group again a GTG set?
\end{question}

\begin{question}\label{LocMin_vs_locGTG}  Is every locally minimal abelian group
necessarily locally GTG? \end{question}

According to Theorem \ref{NewThm1}, for a negative answer to Question
\ref{LocMin_vs_locGTG} it suffices to build a non-discrete locally minimal
group topology on an infinite bounded abelian group of size $< \cont$.
To emphasize  better the situation let us formulate this question in the
following  very specific case:

\begin{question}\label{LocMin_vs_locGTG2} Does the infinite Boolean group
$\bigoplus_\omega \Z(2)$ admit a non-discrete locally minimal  group topology?  A positive answer to this question implies a negative answer to Question
 \ref{LocMin_vs_locGTG}.
\end{question}

Actually, the following weaker version of Question
\ref{LocMin_vs_locGTG} will still be useful for Theorem
\ref{GTG+locMin}:

\begin{question}\label{LocMin_vs_locGTG1} If $G$ is a $U$-locally minimal
abelian group for some $U \in
{\mathcal V}(0),$ does there exist a GTG neighborhood of $0$ contained in $U$?
\end{question}

Theorem \ref{MJX} suggests also another weaker version of Question
\ref{LocMin_vs_locGTG}:

\begin{question}\label{LocMin_vs_locGTG3}  Is every locally minimal NSS abelian
group necessarily locally GTG? \end{question}A positive answer to this question will modify the equivalence proved in Theorem
\ref{MJX} to equivalence between UFSS and  the conjunction of local minimality
and NSS.

\begin{remark}\label{elecero_nolocmin} Proposition \ref{minsubgr} shows that the space in Example
\ref{elecero} cannot provide an answer to Question
\ref{LocMin_vs_locGTG}, since actually it is not locally minimal.
[Suppose that for some $\varepsilon\in (0,1),$ $W_{\varepsilon}$
witnesses local minimality of $L^0.$ Let $h$ be the characteristic function
of $[0,\varepsilon/2],$ and $H$ the subgroup $\langle h
\rangle$ of $L^0.$ $H$ is discrete, hence it cannot be minimal; however,
$H\subseteq W_{\varepsilon/2}$ and thus
$H+W_{\varepsilon/2}\subset W_{\varepsilon},$ which contradicts Prop.
\ref{minsubgr}.]
\end{remark}

\begin{question}\label{ques:NSS}
Is every locally minimal NSS group metrizable?
According to Proposition \ref{loc_min+NSS>metr}, this is true for abelian groups.\end{question}

%\NB \footnote{Here used to be Open question 5.8 resolved in the strongets possible way by Lydia's
%powerful construction. D.D.}

 The next question is related to Proposition \ref{subgr}  and Corollary \ref{Coro_}:

\begin{question}\label{LAST_QUES}
Let $H$ be a closed subgroup of a (locally) minimal group $G$. Is then $H$ necessarily locally minimal ?

%%%%%%%%%%%%%%%%%%%%%%%%%%%%
\end{question}

\end{document}